\documentclass[a4paper, 10pt]{article}

\usepackage[cp1251]{inputenc}
\usepackage[english]{babel}
\usepackage{amsthm}
\usepackage{cite}
\usepackage{tikz-cd}
\usepackage{amsmath}
\usepackage{amsfonts}
\usepackage{amssymb}
\usepackage{hyperref}
\usepackage[]{authblk}
\usepackage{epsfig,graphicx}
\usepackage{indentfirst}

\DeclareMathOperator{\Aut}{Aut}

\DeclareMathOperator{\End}{End}
\DeclareMathOperator{\Char}{char}

\DeclareMathOperator{\Fr}{Fr}
\DeclareMathOperator{\ad}{ad}
\DeclareMathOperator{\Deg}{deg}
\DeclareMathOperator{\Der}{Der}

\DeclareMathOperator{\Id}{Id}
\DeclareMathOperator{\Ht}{ht}

\DeclareMathOperator{\GL}{GL}
\newcommand{\TAut}{\operatorname{TAut}}
\DeclareMathOperator{\Ind}{Ind}

\newtheorem{thm}{Theorem}[section]
\newtheorem{lem}[thm]{Lemma}

\newtheorem{prop}[thm]{Proposition}
\newtheorem{cor}[thm]{Corollary}
\newtheorem{conj}[thm]{Conjecture}
\newtheorem{Def}[thm]{Definition}
\newtheorem{remark}[thm]{Remark}

\begin{document}
\fontsize{11}{11pt}\selectfont
\renewcommand{\thefootnote}{\fnsymbol{footnote}}
 \footnotetext{\emph{Key
words:} Weyl algebra automorphisms, polynomial symplectomorphisms, deformation quantization,
infinite prime number.}
\renewcommand{\thefootnote}{\arabic{footnote}}
\fontsize{12}{12pt}\selectfont
\title{\bf Independence of the B-KK Isomorphism of \\
Infinite Prime}
\renewcommand\Affilfont{\itshape\small}

\author[1]{Alexei Belov-Kanel\thanks{kanel@mccme.ru}}
\author[2,3]{Andrey Elishev\thanks{elishev@phystech.edu}}
\author[1]{Jie-Tai Yu\thanks{jietaiyu@szu.edu.cn}}

\affil[1]{College of Mathematics and Statistics, Shenzhen University, Shenzhen, 518061, China}
\affil[2]{Laboratory of Advanced Combinatorics and Network Applications, Moscow Institute of
Physics and Technology, Dolgoprudny, Moscow Region, 141700, Russia}
\affil[3]{Department of Innovations and High Technology, Moscow Institute of Physics and
Technology, Dolgoprudny, Moscow Region, 141700, Russia}

\date{}

\maketitle
\renewcommand{\abstractname}{Abstract}
\begin{abstract}
We investigate a certain class of $\Ind$-scheme morphisms corresponding to homomorphisms between
the automorphism group of the $n$-th complex Weyl algebra and the group of Poisson
structure-preserving automorphisms of the commutative complex polynomial algebra in $2n$
variables. A conjecture of Kanel-Belov and Kontsevich, whose proof we have recently obtained \cite{K-BE4},
states that these automorphism groups are canonically isomorphic in characteristic zero, with the
mapping discussed here being the candidate for the isomorphism. The main objective of the present
paper is to establish the independence of the said mapping of the choice of infinite prime - that
is, the class $[p]$ of prime number sequences modulo fixed non-principal ultrafilter
$\mathcal{U}$ on the index set of positive integers. To that end, we introduce the augmented and
skew augmented versions of algebras in question and study the augmented $\Ind$-morphism between
the normalized automorphism $\Ind$-schemes in the context of tame automorphism approximation. In
order to correctly implement approximation in our proof, we study singularities of curves in skew
augmented automorphism $\Ind$-schemes and their images under Ind-scheme morphisms. Apart from
that, we study the augmented version of the independence conjecture.
\end{abstract}
\section{Introduction}
\subsection{Overview}
The celebrated Jacobian conjecture $JC_n$ states that every polynomial endomorphism
$$
\varphi: \mathbb{K}[x_1,\ldots, x_n]\rightarrow\mathbb{K}[x_1,\ldots, x_n]
$$
($\mathbb{K}$ has characteristic zero) with constant non-zero Jacobian is in fact an automorphism.
It is a hard open problem which, since it was posed by Keller in 1939 \cite{Keller}, has stimulated
a great deal of progress in areas of algebraic geometry, ring theory, and general algebra
\cite{Yag1, Yag2, Di, DiLev, BBRY, Umir1}.

\smallskip

One of the major developments concerning the Jacobian conjecture was provided by a connection with
a classical problem of J. Dixmier \cite{1}. The Dixmier conjecture is formulated as follows. Let
$W_{n,\mathbb{C}}$ denote the $n$-th Weyl algebra over the field of complex numbers
\begin{equation*}
W_{n,\mathbb{C}}=\mathbb{C}\langle x_1,\ldots,x_n,y_1,\ldots,y_n\rangle/(x_ix_j-x_jx_i,\; y_iy_j-y_jy_i,\; y_ix_j-x_jy_i-\delta_{ij}),
\end{equation*}
($\delta_{ij}$ is the Kronecker symbol). The Dixmier conjecture ($DC_n$) states:
\begin{conj}[Dixmier conjecture] \label{dc}
Every algebra endomorphism of $W_{n,\mathbb{C}}$ is invertible.
\end{conj}
It has been known (see, for instance, \cite{Bass} or \cite{3}) that $DC_n$ implies $JC_n$.
Tsuchimoto \cite{13},  and independently Kanel-Belov and Kontsevich \cite{3}, have discovered the
following converse implication.
\begin{thm}[Tsuchimoto \cite{13}, Kanel-Belov and Kontsevich \cite{3}]\label{dcjc}
$JC_{2n}$ implies $DC_n$.
\end{thm}
From this theorem it immediately follows that the conjunctions $JC_{\infty}$ and $DC_{\infty}$ are
equivalent -- or, as is usually said, that the Jacobian conjecture is stably equivalent to the
Dixmier conjecture.

\smallskip

The cornerstone of the proof of Theorem \ref{dcjc} (both in \cite{3} and \cite{13}) is the
construction of a semigroup homomorphism
$$
\phi_{[p]}: \End W_{n,\mathbb{C}}\rightarrow \End P_{n,\mathbb{C}},
$$
where $P_{n,\mathbb{C}}$ denotes the so-called {\it commutative Poisson algebra}, which is the
algebra
$$
\mathbb{C}[x_1,\ldots, x_n,p_1,\ldots, p_n]
$$
of commutative polynomials in $2n$ variables equipped with the additional Poisson bracket
$$
\lbrace x_i,x_j\rbrace = 0,\;\;\lbrace p_i,p_j\rbrace=0,\;\;\lbrace p_i,x_j\rbrace=\delta_{ij}.
$$
The key property of the homomorphism $\phi_{[p]}$ is that it restricts to a group homomorphism
$$
\phi_{[p]}: \Aut W_{n,\mathbb{C}}\rightarrow \Aut P_{n,\mathbb{C}}
$$
between the automorphism groups of the Weyl and Poisson algebras (i.e. the images of automorphisms
under $\phi_{[p]}$ are automorphisms). The consequence of this fact is that a counterexample to the
Dixmier conjecture $DC_n$ must be mapped by $\phi_{[p]}$ to a polynomial endomorphism (in $2n$
variables) with constant non-zero Jacobian which is not an automorphism.

\smallskip

The nature of the homomorphism $\phi_{[p]}$, constructed both in \cite{13} and \cite{4,3} as well
as in the main part of this paper, hints at similarities between the automorphism groups $\Aut
W_{n,\mathbb{C}}$ and $\Aut P_{n,\mathbb{C}}$. In connection with this circumstance, the following
conjecture, due to Kanel-Belov and Kontsevich (often dubbed here and elsewhere the Kontsevich
conjecture), was formulated and studied \cite{4}.
\begin{conj}[Kontsevich conjecture]\label{mainconjgen}
The automorphism groups
$$
\Aut W_{n,\mathbb{K}}\;\;\text{and}\;\;\Aut P_{n,\mathbb{K}}
$$
are isomorphic in characteristic zero.
\end{conj}

In particular, in view of the connection to the Jacobian and Dixmier conjectures, this conjecture
can be slightly modified as follows.
\begin{conj}\label{mainconj}
If $\mathbb{K}=\mathbb{C}$, then
$$
\phi_{[p]}: \Aut W_{n,\mathbb{C}}\rightarrow \Aut P_{n,\mathbb{C}}
$$
is an isomorphism.
\end{conj}

\smallskip

The construction of the homomorphism $\phi_{[p]}$ is quite elaborate and involves a fixed
non-principal ultrafilter on the index set as well as the so-called \emph{infinite prime numbers},
denoted by $[p]$ -- elements of the ring of hyperintegers (constructed with respect to the fixed
ultrafilter). In the form given in Tsuchimoto's paper \cite{13}, it also requires the ground field
to be algebraically closed, which, in accordance with the Lefschetz principle, allows one to set
$\mathbb{K}=\mathbb{C}$.

\smallskip

In \cite{4}, several generalizations of Conjecture \ref{mainconjgen} were studied. The most general
reformulation has to do with holonomic modules over the Weyl algebra (holonomic
$\mathcal{D}$-modules) and is stated as follows.
\begin{conj}\label{conjhol}
There is a one-to-one correspondence between irreducible holonomic $\mathcal{D}$-modules over $W_n$
and lagrangian subvarieties of the affine space (of corresponding dimension).
\end{conj}

One direction in the Conjecture \ref{conjhol} -- namely the construction of a lagrangian subvariety
from a given holonomic module -- has been accomplished by Bitoun \cite{Bit} and, independently, Van
den Bergh \cite{VdB}, who gave a conceptually different proof. Also, the one-dimensional case of
Conjecture \ref{conjhol} was studied in \cite{K-BE}.


Dodd \cite{Dodd} has established a number of far-reaching results of homological nature which, as
far as our understanding is, imply a version of Conjecture \ref{conjhol} (cf. Theorem 1, Corollary
2 and Theorem 3 of \cite{Dodd}). His argument is based on properties of the so-called $p$-support,
defined by Kontsevich in \cite{Kon2}.

\smallskip

Conjecture \ref{mainconjgen} is positive for $n=1$. This is a classical result due to Jung
\cite{Jung} (cf. also \cite{VdK}) and Makar-Limanov \cite{8,9}, whose proof is essentially a
description of the respective automorphism groups. We have recently put forward a proof of
Conjecture \ref{mainconj}, \cite{K-BE4}, which (unlike that of Dodd, which uses ideas of
homological nature) relies on a certain deformation (or augmentation) of the Poisson structure as
well as on approximation by tame symplectomorphisms (the meaning of these terms is explained later
in the introduction as well as in the main text).

\smallskip

In \cite{K-BE4}, it is demonstrated that the group homomorphism $\phi_{[p]}$ is one to one (for any
arbitrary fixed infinite prime $[p]$). Once this is done, the most pertinent question becomes
whether the dependence of this isomorphism on the infinite prime $[p]$ can be eliminated. In other
words, the closest companion to the main Conjecture \ref{mainconj} is the following statement:
\begin{conj}\label{mainconjindep}
The homomorphism $\phi_{[p]}$ \footnote{The Kanel-Belov -- Kontsevich homomorphism, sometimes
dubbed B-KK homomorphism.} is independent of the choice of the infinite prime $[p]$.
\end{conj}



This paper focuses on Conjecture \ref{mainconjindep}.  Namely, we will prove, \emph{assuming
Conjecture} \ref{mainconj} (which is justified in \cite{K-BE4}), that for any two distinct infinite
primes $[p]$ and $[p']$ the composition
$$
\phi_{[p]}\circ\phi_{[p']}^{-1}:\Aut(P_{n,\mathbb{C}})\rightarrow \Aut(P_{n,\mathbb{C}})
$$
is the identity map.

\smallskip

The so-called \emph{loop morphism} $\phi_{[p]}\circ\phi_{[p']}^{-1}$, together with its obvious
counterpart for the Weyl algebra, provides an example of an $\Ind$-morphism of $\Ind$-schemes. The
meaning of this is as follows.

Denote, for the sake of brevity, by $\zeta_i,\;i=1,\ldots,2n$ the standard generators $x_j,\;d_i$
of the Weyl algebra. If these carry degree one, then there is the obvious filtration of $W_n$ by
the total degree. This filtration induces a filtration on the automorphism group:
\begin{equation*}
\Aut^{\leq N} W_{n,\mathbb{C}}:=\lbrace f\in\Aut(W_{n,\mathbb{C}})\;|\;\Deg f(\zeta_i)\leq N,\forall i=1,\ldots,2n\rbrace.
\end{equation*}

Analogous sets of the form $\Aut^{\leq N}$ are defined for the symplectomorphism group
$\Aut(P_{n,\mathbb{C}})$. The sets $\Aut^{\leq N}$ are in fact affine algebraic sets: indeed, any
algebra automorphism is the same as the collection of images of algebra generators with respect to
that automorphism; those in turn are linear combinations of monomials which constitute the standard
basis. The coefficients in these linear combinations are constrained by polynomial equations (which
ensure invertibility together with the preservation of the relevant algebraic structures, such as
the commutator or the Poisson bracket) and therefore can be thought of as coordinate functions for
an algebraic set in the affine space of suitable dimension.

\smallskip

Whether these algebraic sets are in fact varieties (i.e. whether they are Zariski-irreducible) is
currently unknown, however at any rate to these sets correspond certain schemes. We take
normalizations of these schemes and denote them also by $\Aut^{\leq N}$\footnote{The normalization
is needed for Proposition \ref{propgabber} and therefore for the main proof to work.}. In spite of
minor abuse of notation, we will denote the algebraic sets together with their normalizations and
schemes corresponding to those all by $\Aut^{\leq N}$.

\smallskip

The maps
\begin{equation*}
\Aut^{\leq N} W_{n,\mathbb{C}}\rightarrow \Aut^{\leq N+1} W_{n,\mathbb{C}}
\end{equation*}
(defined in an obvious manner for the algebraic sets) are Zariski-closed embeddings, and the entire
group $\Aut (W_{n,\mathbb{C}})$ is a direct limit of the inductive system formed by $\Aut^{\leq N}$
together with these maps. The same can be said for the symplectomorphism group
$\Aut(P_{n,\mathbb{C}})$.

The scheme version of this direct system of algebraic sets and closed embeddings, together with its
direct limit, constitutes what is known as an $\Ind$-\emph{scheme}. The origin of the concept, as
well as the relevant definitions, is due to Shafarevich \cite{Shafarevich}. An $\Ind$-morphism of
$\Ind$-schemes (or $\Ind$-scheme morphism) is a continuous\footnote{The base spaces of direct
limits are endowed with the direct limit topology, cf. \cite{Shafarevich}.} mapping of
$\Ind$-schemes which, for every object in the domain direct system, restricts to a morphism between
that object and a corresponding object of the codomain direct system. In this sense, the "loop
morphisms" of this paper are $\Ind$-scheme morphisms.

\smallskip

Of course, the same definitions could be repeated for the endomorphism semigroups $\End$, which
means that the latter correspond to $\Ind$-schemes (also denoted by $\End$) as well.

\smallskip

It is therefore sensible, in the context of Conjecture \ref{mainconj}, to study the geometry of
$\Ind$-schemes corresponding to groups of automorphisms of the relevant algebras; in particular,
the properties of automorphisms of such $\Ind$-schemes are of main interest. This subject has its
origins in the classical work of B. I. Plotkin \cite{BIP1, BIP2} (that study was conducted in the
realm of general algebra and therefore was largely unrelated to the relatively modern particulars
discussed here).

\smallskip

The description of $\Ind$-automorphisms of $\Ind$-schemes -- that is, the study of spaces of the
form $\Aut\Aut$ and $\Aut\End$ -- has been done in various instances. In the case of $\End$, it
turned out to be possible \cite{KBL, KBLBerz, Berz} to describe the entire groups $\Aut\End$ (i.e.
not just the subgroup of mappings preserving the $\Ind$-scheme structure) for the cases of free
associative and commutative polynomial algebras. On the other hand, for the same algebras the
spaces $\Aut_{\Ind}\Aut$ for the same algebras were systematically studied in \cite{KBYu}. In that
study, tame approximation as well as a certain singularity technique (explained and utilized in
this paper as well) has been shown to be rather helpful.

\smallskip

In the case of the Poisson algebra $P_n$, we have not been able to obtain the description of
$\Aut_{\Ind}\Aut$ by means of tame approximation, as in the easier case of commutative polynomial
algebra in \cite{KBYu}. Certain topological properties are missing. One way around this is the
introduction of deformation (or augmentation) of the algebra and its Poisson structure by a central
variable $h$, so that
$$
\lbrace p_i, x_j\rbrace = h\delta_{ij}
$$
in the new algebra. This augmentation modifies (homogenizes) the Poisson structure to match the
formal power series topology which defines tame approximation. Furthermore, the Poisson structure
can be further distorted by allowing non-zero commutation of distinct generators, in such a way as
to make the morphisms in question continuous. This crucial idea is what enables us to resolve the
symplectomorphism lifting problem in \cite{K-BE4} which leads to the proof of Conjecture
\ref{mainconj}. The tradeoff of the augmentation approach is the need to \emph{specialize} the new
variable $h$ (to $h = 1$) in order to return to the original conjecture for non-augmented algebras.
This is achieved, both here and in \cite{K-BE4}, by means of a certain homotopy argument. The
procedure is {\bf not at all trivial } and requires, for its last step to work, the invertibility
of mappings corresponding to points in $\Aut$ (that is, it does not work with endomorphisms).

\smallskip

In order to obtain the stronger topological properties of tame approximation (continuity of the
loop morphism in the skew augmented case), we use a certain \textbf{singularity trick}. In the broadest
terms, it is a proof technique that allows one, by examining orders of singularities of certain
curves and their images under the studied morphism, to obtain useful data on the morphism. The
technique was utilized in \cite{KBYu} as well as in the main text. It is not applicable in the
non-augmented case.

\smallskip

The study of loop morphisms as $\Ind$-automorphisms is a meaningful effort in its own right, given
its conceptual connection with the methods of \cite{KBYu}. As it turns out, the loop morphism
$\phi_{[p]}\circ\phi_{[p']}^{-1}$ belongs to a certain class of asymptotic expansion preserving
morphisms, for which local (strata-wise) unipotency can be established. This result, significant
primarily in the context of the study \cite{KBYu}, will be given and proved in the accompanying
paper \cite{K-BEadd}. The description of $\Aut_{\Ind}\Aut P_n$ along the lines of \cite{KBYu} is
currently out of reach; nevertheless, something regarding the space  $\Aut_{\Ind}\Aut P_n$ can be
obtained by means of straightforward geometric considerations.

\smallskip



The construction of the morphism $\phi_{[p]}$, as well as the notion of infinite prime in the
context of algebraic geometry, has its roots in model theory. Application of model theory to
problems in algebraic geometry has been developed by E. Hrushovski and B. Zilber, cf. for instance
\cite{BH} and \cite{HZ}.

The study of $\Ind$-schemes using tools from model theory is also our objective, with Conjecture
\ref{mainconj} being one of the more significant examples. It is our hope that the investigation of
such instances will yield new insights into the theory itself.

\subsection{Main results}
The main results of the present paper are as follows. First and foremost, we give a proof of
Conjecture \ref{mainconjindep}, in the form of the following Main Theorem.

\begin{thm}[Main Theorem] \label{maintheorem1}
Assuming Conjecture \ref{mainconj}, for any two infinite primes $[p],[p']$ the loop morphism
$$
\Phi = \phi_{[p]}\circ\phi_{[p']}^{-1}
$$
and its symplectic counterpart $\Phi_s$ is the identity map.
\end{thm}
Evidently, it does not matter much which one of the loop morphisms -- $\Phi$ or $\Phi_s$ -- is to
be proved trivial: once the fact is established for any one of them, the statement for the other
one follows at once. We will therefore often switch from one type of morphism to the other for the
sake of local convenience. Also, the statement of the Main Theorem is trivial for $[p] = [p']$,
therefore we will assume $[p]\neq [p']$ in our construction of $\Phi$.

\smallskip

The proof of the Main Theorem relies on the properties of the so-called augmented (or deformed)
versions of the Weyl and Poisson algebras, which are introduced in the next subsection, as well as
utilizes augmented tame automorphism approximation, as in Section 4.

\smallskip

The key step in the establishment of the Main Theorem is the proof that the augmented version of
the loop morphism is the identity map. Namely, we have the following.
\begin{thm}[Main Theorem 2]\label{maintheorem4}
The morphism
$$
\Phi^h_s: \Aut P^h_{n,\mathbb{C}}\rightarrow \Aut P^h_{n,\mathbb{C}}
$$
given by the composition $(\phi^h)_{[p]}\circ(\phi^h)_{[p']}^{-1}$, for any fixed pair $[p],[p']$,
is the identity map.
\end{thm}
Again, we are obviously interested in the non-trivial case $[p]\neq [p']$.

The statement in the main text which corresponds to this theorem is Proposition \ref{mainprop}
(formulated for the augmented Weyl algebra, which is not significantly different from the statement
of the Main Theorem 2 as we have remarked before). We note that Main Theorem 2 is the augmented
analogue of the independence of infinite prime conjecture (Conjecture \ref{mainconjindep} and the
Main Theorem). As formulated above, it requires the assumption of the augmented version of
Conjecture \ref{mainconj}. The justification of this assumption is provided in the proof of
Conjecture \ref{mainconj} in our paper \cite{K-BE4}.

\smallskip

We conclude the overview of the main results with the following observation. Most of our analysis
takes place over algebraically closed base field -- in particular, Tsuchimoto's construction of the
homomorphism $\phi_{[p]}$ leads to an algebraically closed "universal" base field of characteristic
zero. However, in \cite{4}, a slightly different, more general construction of a group homomorphism
over the rationals is presented -- namely, it is proved in \cite{4} that for $R$ a commutative ring
there exists a (unique) group homomorphism
$$
\phi_R: \Aut W_{n,R}\rightarrow \Aut P_{n,R_{\infty}}
$$
where
\begin{equation*}
R_{\infty}=\lim_{\rightarrow}\left( \prod_{p} R'\otimes \mathbb{Z}/p\mathbb{Z}\;/\;\bigoplus_{p} R'\otimes \mathbb{Z}/p\mathbb{Z}\right),
\end{equation*}
is the reduction modulo infinite prime. The image  $\phi_R(f)$ is essentially (untwisted by the
Frobenius morphism) a collection of restrictions to centers of Weyl algebras over $\mod\;p$
reductions of a finitely generated subring $R'\subset R$ over which the automorphism $f$ is
defined. This construction is supposed to be the replacement of $\phi_{[p]}$ for the case of
arbitrary base field of characteristic zero (or equivalently the field $\mathbb{Q}$ of rational
numbers). In light of that, one has the following statement (Conjecture 3 of \cite{4}):
\begin{conj}\label{conj3}
The image of $\phi_R$ belongs to
$$
\Aut P_{n,i(R)\otimes \mathbb{Q}}
$$
where $i:R\rightarrow R_{\infty}$ denotes the tautological inclusion.
\end{conj}
From that Conjecture a collection of constructible maps
$$
\phi_{n,N}: \Aut^{\leq N}W_{n,\mathbb{Q}}\rightarrow \Aut^{\leq N}P_{n,\mathbb{Q}}
$$
may be defined, as proved in \cite{4}. These maps serve as generalizations of $\phi_{[p]}$ (dealt
with in this paper) and constitute the conjectured canonical isomorphism of Conjecture
\ref{mainconjgen}. What this means for the ultrafilter construction considered here is that \emph{a
priori} it is not known whether $\phi_{[p]}$ is defined over the rationals. However -- and this is
the point of our observation -- the introduction of augmentation and augmented tame approximation
should allow us to circumvent this obstacle. The line of reasoning is as follows.

\smallskip

Firstly, the isomorphism of tame subgroups (Theorem 1 of \cite{4}) continues to be valid for the
base field $\mathbb{Q}$. Secondly, both in augmented and non-augmented cases, tame approximation
requires only characteristic zero to work, and therefore is valid over the rationals. In the
augmented case, however, tame approximation is stronger in the sense that the tame isomorphism is
continuous in a neighborhood of the identity map. Therefore, for Weyl $\mathbb{Q}$-algebra
automorphisms sufficiently close to the identity one can take a well defined limit of the image
under the tame isomorphism of an arbitrary converging tame sequence. Thus defined mapping can be
extended to the whole space (in a manner similar to that of Section 4) to yield a mapping which
must coincide with the augmented version of $\phi_R$ (for $R=\mathbb{Q}$). One then specializes
(again as in Section 4) the augmentation parameters to return to the non-augmented case. Most of
the mechanics of Section 4, as well as the proof of Conjecture \ref{mainconj} in \cite{K-BE4} is
adapted to this situation. Thus one should be able to obtain the proof of Conjecture
\ref{mainconjgen} by augmented tame approximation and specialization, together with its canonicity
(since the tame isomorphism is in fact independent of infinite prime). In other words, we must have
the following

\begin{thm}[Main Theorem 3]\label{mainthmgen}
The groups $\Aut W_{n,\mathbb{Q}}$ and $\Aut P_{n,\mathbb{Q}}$ are canonically isomorphic.
\end{thm}

\smallskip

Therefore, the suitable (rather straightforward) adaptation of the augmented tame approximation
allows one to obtain the proof of the general form of Kontsevich conjecture on the isomorphism --
Conjecture \ref{mainconjgen}.

\smallskip

It is also easy to see that the construction above, together with the proof in Section 4 (as well
as the proof of the isomorphism conjecture in \cite{K-BE4}) applies to the case of \emph{any} base
field of characteristic zero. Therefore, the Main Theorems 1 and 2 admit the following
generalization.

\begin{thm}[Main Theorem 4]\label{mainthmgen}
Let $\mathbb{K}$ be a field of characteristic zero. Then the groups $\Aut W_{n,\mathbb{K}}$ and
$\Aut P_{n,\mathbb{K}}$ are isomorphic, with the isomorphism being provided by any morphism
$\phi_{[p]}$ as constructed above.
\end{thm}

\begin{thm}[Main Theorem 5] \label{maintheorem1a}
For any two infinite primes $[p],[p']$ and arbitrary base field $\mathbb{K}$, $\Char\mathbb{K} =
0$, the morphisms
$$
\phi_{[p]},\phi_{[p']}:\Aut W_{n,\mathbb{K}}\rightarrow \Aut P_{n,\mathbb{K}}
$$
(defined as above) coincide.
\end{thm}

\begin{thm}[Main Theorem 6]\label{maintheorem4}
For any two infinite primes $[p],[p']$ and any base field $\mathbb{K}$, $\Char\mathbb{K} = 0$, the
morphisms
$$
\phi^h_{[p]},\phi^h_{[p']}:\Aut W^h_{n,\mathbb{K}}\rightarrow \Aut P^h_{n,\mathbb{K}}
$$
of the augmented algebras coincide.
\end{thm}

\smallskip

While these considerations are reasonable, the detailed exposition of proofs lies beyond the scope
of the present paper. We will focus on these problems in our further work.

\subsection{Definitions and preliminaries}

\subsubsection{Tame and wild automorphisms}
The proof of the Main Theorem, as well as the strengthened version of Conjecture \ref{mainconj}
(cf. our newest paper \cite{K-BE4}), relies on the phenomenon of approximation by tame
automorphisms. The definition is as follows.

Suppose first that $\mathbb{K}[x_1,\ldots, x_n]$ is the polynomial algebra over a field
$\mathbb{K}$, and let $\varphi$ be an automorphism of this algebra.
\begin{Def} \label{defelement} $\varphi$ is an \textbf{elementary} automorphism if it is of the form
\begin{equation*}
\varphi = (x_1,\ldots,\;x_{k-1},\;ax_k+f(x_1,\ldots,x_{k-1},\;x_{k+1},\;\ldots,\;x_n),\;x_{k+1},\;\ldots,\;x_n)
\end{equation*}
with $a\in\mathbb{K}^{\times}$.
\end{Def}
Observe that linear invertible changes of variables -- that is, transformations of the form
\begin{equation*}
(x_1,\;\ldots,\;x_n)\mapsto (x_1,\;\ldots,\;x_n)A,\;\;A\in\GL(n,\mathbb{K})
\end{equation*}
are realized as compositions of elementary automorphisms.
\smallskip
\begin{Def} \label{deftame}
A \textbf{tame} automorphism is an element of the subgroup \\
$\TAut \mathbb{K}[x_1,\ldots,x_n]$ generated by all elementary automorphisms. Automorphisms that
are not tame are called \textbf{wild}.
\end{Def}

All automorphisms of $\mathbb{K}[x,y]$ are tame \cite{Jung, VdK}; there are examples of wild
automorphisms of $\mathbb{K}[x,y,z]$ (Nagata automorphism). It is unknown whether there are wild
automorphisms in the case of $2n$ ($n>1$) generators. An open conjecture asserts that every
polynomial automorphism becomes tame after adjunction of a finite number of variables on which its
action is extended by the identity map.

\smallskip

A tame polynomial symplectomorphism is a tame automorphism that preserves the Poisson structure --
in other words, the group $\TAut P_n$ is the intersection of the tame automorphism group with $\Aut
P_n$. With the standard Poisson structure fixed, it is easy to see the necessary and sufficient
conditions for elementary automorphisms to be symplectic.

\smallskip

Similarly, tame automorphisms of the Weyl algebra are defined.

\smallskip

One of the crucial properties of the homomorphism $\phi_{[p]}$, proved in \cite{4}, is as follows.
\begin{prop}\label{proptame}
  The restriction of $\phi_{[p]}$ to the subgroup $\TAut W_{n,\mathbb{C}}$ is an isomorphism
  $$
  \TAut W_{n,\mathbb{C}}\rightarrow \TAut P_{n,\mathbb{C}}
  $$
\end{prop}

Another important result is the symplectic version of a classical theorem of D. Anick (\cite{An}),
which states that the tame automorphism subgroup is dense in the automorphism group in the formal
power series topology. The proof is in \cite{KGE}.
\begin{thm} \label{appthm}
The subgroup $\TAut P_{n,\mathbb{K}}$ ($\Char \mathbb{K} = 0$) is dense in $\Aut P_{n,\mathbb{K}}$
in the power series topology.
\end{thm}

Approximation by tame automorphisms is essentially Anick's theorem and, in our context, its
symplectic analogue (Theorem \ref{appthm}). Given a symplectomorphism, one may choose a sequence of
tame symplectomorphism which converges to it in the power series topology. Then, using the tame
isomorphism of Proposition \ref{proptame}, one may form a sequence of tame automorphisms of the
corresponding Weyl algebra; one can then take its formal limit (whose action is given by power
series in the Weyl generators) and ask whether this limit is well defined (independent of the
choice of tame symplectomorphism sequence) and exists in $\Aut W_n$. If one manages to prove the
correctness and power series truncation, then the inverse to the homomorphism $\phi_{[p]}$ is
constructed and Conjecture \ref{mainconj} is proved. The tame approximation method is the main
subject of \cite{K-BE4}. As it turns out, the simplest way to construct the inverse (the lifting
map) is to introduce an augmentation of the Poisson (and Weyl commutator) structure by adjoining a
central variable, which is analogous to the deformation parameter (Planck's constant) in the
Kontsevich's quantization recipe and which distorts the power series topology, construct a well
defined lifting map for the new algebras and then show that a specialization of the augmentation
variables (required to return to the non-augmented case of Conjecture \ref{mainconj}) is valid.

\smallskip

\subsubsection{Approximation and the singularity trick}

As it turns out (Proposition \ref{skewthetaprop}), the augmented version of the morphism $\Phi$
corresponds to a morphism of the automorphism group of the skew augmented algebra, which behaves
well with respect to the power series topology; it is this fact (together with tame approximation,
or rather with the fact that $\Phi$ is the identity map on the tame automorphisms) that allows us
to prove the key Proposition \ref{mainprop} from which the Main Theorem follows by specialization.

The proof of Propositions \ref{skewthetaprop} and \ref{mainprop} utilizes a certain "singularity
trick". Essentially it is a technique that, by examining certain curves in $\Aut$, allows one to
efficiently control the height of the higher-degree terms in an automorphism and its image under
$\Phi$ which is near the identity automorphism. In the main text, the two statements that comprise
this technique are Lemma \ref{lem1} and Proposition \ref{singtrick}.

\smallskip

The idea of Poisson (and Weyl) structure augmentation, which enables the proof of both Main Theorem
of this paper and Conjecture \ref{mainconj} in \cite{K-BE4}, and the singularity trick complement
each other and {\bf constitute the cornerstone of our approach to the polynomial symplectomorphism
quantization (lifting) problem} by way of enabling a stronger form of tame approximation. This
advantage is offset by the need to specialize the deformation parameters (Planck constant) in order
to return to the non-augmented case, as well as by the necessity to introduce the general form of
commutation relations (which we call \textbf{skew} augmented algebra structure, owing to its
antisymmetry) however the specialization can be proven correct, as in the last part of the proof of
the Main Theorem as well as in the proof in \cite{K-BE4}.

\smallskip

The singularity trick is rather useful when dealing with direct systems of varieties equipped with
the power series topology. It was first introduced in our prior paper \cite{KBYu} (Theorem 3.2,
Lemma 3.5, 3.6 and 3.7). In our proof, this technique is employed in an essentially the same
manner, and is contained in Lemma \ref{lem1} and Proposition \ref{singtrick}.

\smallskip

For the archetypal case of the automorphism group of the commutative polynomial algebra
$\mathbb{K}[x_1,\ldots, x_n]$, the situation is as follows.

Let $L=L(t)$ be a curve of linear automorphisms, i.e. a curve
$$L \subset \Aut(K[x_1,\dots,x_n]),$$ whose points are linear substitutions. Suppose that, as $t$ tends to zero, the $i$-th eigenvalue of the matrix $L(t)$ (corresponding to the linear changes of variables) also tends to zero as $t^{k_i}$, $k_i\in\mathbb{N}$. Such a family always exists.

Suppose now that the degrees $\lbrace k_i,\;i=1,\ldots n\rbrace$ of singularity of eigenvalues at
zero are such that for every pair $(i,j)$, if $k_i\neq k_j$, then there exists a positive integer
$m$ such that
$$
\text{either\;\;} k_im\leq k_j\;\;\text{or\;\;}k_jm\leq k_i.
$$

The largest such $m$ we will call the \textbf{order} of $L(t)$ at $t=0$. As $k_i$ are all set to be
positive integer, the order equals the integer part of $\frac{k_{\text{max}}}{k_{\text{min}}}$.

Let $M\in \Aut_0(K[x_1,\dots,x_n])$ be a polynomial automorphism.
\begin{lem}    \label{Lm2} The curve $L(t)ML(t)^{-1}$ has no singularity at zero for any $L(t)$ of order $\leq N$ if and only if $M\in \hat{H}_N$, where $\hat{H}_N$ is the subgroup of automorphisms which are homothety modulo the $N$-th power of the
augmentation ideal $(x_1,\ldots, x_n)$.
\end{lem}

The proof can be found in \cite{KBYu}. Also, a direct analogue of this lemma is Proposition
\ref{singtrick}, with proof borrowing its essential features from \cite{KBYu}.

\begin{remark}\label{weakmotivation}
The emergence of the singularity trick (together with the general consideration of singularities of
$\Ind$-schemes) may be viewed as a reflection of the infinite-dimensional nature of the problem.
Indeed, in the context of finite-dimensional algebraic groups, the most natural approach to a
problem such as Conjecture \ref{mainconj} would be the construction of a morphism which induces an
isomorphism of the Lie algebras. However, as was pointed out in \cite{4} (specifically in Section 3
and also in Remark 2 of \cite{4}), the naive infinite-dimensional translation of this approach is
unsatisfactory. In fact, the Lie algebras (defined as the algebras of derivations) of
$W_{n,\mathbb{Q}}$ and $P_{n,\mathbb{Q}}$ are not isomorphic to each other. Therefore, the
positivity of Conjecture \ref{mainconj} leads to the breakdown of the Lie algebra isomorphism
approach typical of finite-dimensional cases. This pathological infinite-dimensional effect could
be the result of the $\Ind$-schemes being singular at every point \cite{4}.

The study of singularities of curves in $\Ind$-schemes offers a viable alternative approach to the
problem. At the technical level, it is a device which allows the efficient handling of the grading
and the induced topology, however some of the finer points of augmented tame approximation are tied
to the singularity trick as well.

\end{remark}

\subsubsection{Augmented and skew Weyl and Poisson algebras}

In the proof of our main result we make use of the augmented (deformed, or quantized) versions of
$W_n$ and $P_n$, which we now define. In order to deform the Weyl algebra $W_n$, we introduce the
augmentation parameter $h$ and modify the commutator between $d$ and $x$ by setting
$$
[d_i, x_j] = h \delta_{ij}.
$$
Alternatively, one can start with the free algebra $\mathbb{K}\langle a_1,\ldots, a_n, b_1,\ldots,
b_n, c\rangle$ and take the quotient with respect to the following set of identities:
$$
a_ia_j-a_ja_i,\;\;b_ib_j-b_jb_i\;\;,
b_ia_j-a_jb_i-\delta_{ij}c,\;\;
a_ic-ca_i,\;\;
b_ic-cb_i.
$$
The quotient algebra, which we denote by $W^h_n$ (or $W^h_{n,\mathbb{K}}$ to indicate the base
field) is the augmented Weyl algebra.

Similarly, we may distort the Poisson bracket of $P_n$:
\begin{equation*}
\lbrace p_i,x_j\rbrace = h\delta_{ij}
\end{equation*}
to reflect the augmentation of $W_n$ into $W^h_n$ in the classical counterpart. (Here we have
renamed the generators $z_i$ into $x_i$ and $p_i$, according to their behavior with respect to the
Poisson bracket. The notation is standard.) The resulting polynomial algebra will be denoted by
$P^h_n$.
\medskip

The algebras $W^h_n$ and $P^h_n$ are connected by an analogue of the Kontsevich homomorphism (which is constructed for the non-augmented case below); the main point is that the independence of this quantized homomorphism of infinite prime implies the independence of infinite prime of the non-augmented morphism. 
A version of these augmented algebras appeared in the work of Myung and Oh \cite{MyOh}.

\smallskip

In order to create a situation in which the power series topology is well respected by the loop
morphisms, we distort the augmentation further by introducing a pair of \textbf{skew augmented}
algebras $W_{n,\mathbb{C}}^h[k_{ij}]$ and $P_{n,\mathbb{C}}^h[k_{ij}]$ (which correspond to the
$h$-augmented Weyl and Poisson algebras, respectively).

These are defined as follows. Let the augmented Poisson generators be denoted by $\xi_i$ with
$1\leq i\leq 2n$ and let $[k_{ij}]$ be an antisymmetric matrix of central variables. The algebra
$P_{n,\mathbb{C}}^h[k_{ij}]$ is generated by $2n$ commuting variables $\xi_i$, the augmentation
variable $h$ and the variables $[k_{ij}]$ (thus being the polynomial algebra in these variables);
the Poisson bracket is defined on the generators $\xi_i$:
$$
\lbrace \xi_i,\xi_j\rbrace = hk_{ij}.
$$
The bracket of any element with $h$ or with any of the $k_{ij}$ is zero. The skew version of the
algebra $W_{n,\mathbb{C}}$ is defined analogously.

It is for these auxiliary algebras that the analogue of the Kontsevich conjecture, together with
the independence of infinite prime, can be proved, as we demonstrate in \cite{K-BE4} as well as in
the present paper. Once that is done, the proof of the main results is finalized by the
specialization argument.

From the standpoint of our context, much of the theory of the Weyl and Poisson algebras remains
unchanged by the augmentation as well as by skew augmentation: the construction of the homomorphism
$\phi_{[p]}$ is identical, the proof of the augmented version of Proposition \ref{propgabber} is
unchanged as well, and the tame augmented symplectomorphism approximation for the $h$-augmented
algebra is established in a way almost identical to the one presented in \cite{KGE}. What does
change is the behavior of $\phi_{[p]}$ and the loop morphisms with respect to the approximation in
the skew case, as the skew augmented version of $\Phi$ will be continuous in the power series topology.
This is the main point of the augmentation.

\subsection{Plan of the paper}

The main body of the paper is divided into three sections (Sections 2 -- 4). In the first section
(Section 2), the construction of the homomorphism $\phi_{[p]}$ is recalled. In this form, the
construction is originally due to Tsuchimoto \cite{13}. A slightly alternative point of view can be
found in \cite{4}.

\smallskip

Section 3 provides the proof of the fact that the maps $\phi_{[p]}$ and $\Phi$ are morphisms.

\smallskip

Finally, Section 4 is dedicated to the proof of the Main Theorem, where the deformed algebras,
augmented tame approximation, and the singularity trick are used to their full extent.

\subsection*{Acknowledgments}

We thank Ivan Arzhantsev, Roman Karasev, David Kazhdan, V. O. Manturov, Eugene Plotkin, Eliyahu
Rips, Ernest Vinberg and Boris Zilber for numerous helpful remarks. We also thank Alexander Zheglov
and Georgy Sharygin, as well as Ilya Karzhemanov and Ilya Zhdanovskii and the participants of their
respective seminars who provided for a most engaging and fruitful discussion of our findings.

This work is supported by the Russian Science Foundation grant No. 17-11-01377.

\section{Infinite primes and the morphism $\phi_{[p]}$}
\subsection{Ultrafilters and infinite primes}
Let $\mathcal{U}\subset 2^{\mathbb{N}}$ be an arbitrary non-principal ultrafilter on the set of all
positive numbers ($\mathbb{N}$ will almost always be regarded as the index set in this note) . Let
$\mathbb{P}$ be the set of all prime numbers, and let $\mathbb{P}^{\mathbb{N}}$ denote the set of
all sequences $p=(p_m)_{m\in\mathbb{N}}$ of prime numbers. We refer to a generic set $A\in
\mathcal{U}$ as an index subset in situations involving the restriction $p_{|A}:A\rightarrow
\mathbb{P}$. We will call a sequence $p$ of prime numbers $\mathcal{U}$-\textbf{stationary} if
there is an index subset $A\in\mathcal{U}$ such that its image $p(A)$ consists of one point.
\smallskip

A sequence $p:\mathbb{N}\rightarrow\mathbb{P}$ is bounded if the image $p(\mathbb{N})$ is a finite
set. Thanks to the ultrafilter finite intersection property, bounded sequences are necessarily
$\mathcal{U}$-stationary.

\smallskip

Any non-principal ultrafilter $\mathcal{U}$ generates a congruence
\begin{equation*}
\sim_{\mathcal{U}}\subseteq \mathbb{P}^{\mathbb{N}}\times\mathbb{P}^{\mathbb{N}}
\end{equation*}
in the following way. Two sequences $p^1$ and $p^2$ are $\mathcal{U}$-congruent iff there is an
index subset $A\in\mathcal{U}$ such that for all $m\in A$ the following equality holds:
\begin{equation*}
p^1_m=p^2_m.
\end{equation*}
The corresponding quotient
\begin{equation*}
{}^*\mathbb{P}\equiv \mathbb{P}^{\mathbb{N}}/\sim_{\mathcal{U}}
\end{equation*}
contains as a proper subset the set of all primes $\mathbb{P}$ (naturally identified with classes
of $\mathcal{U}$-stationary sequences), as well as classes of unbounded sequences. The latter are
referred to as nonstandard, or infinitely large, primes. We will use both names and normally denote
such elements by $[p]$, mirroring the convention for equivalence classes. The terminology is
justified, as the set of nonstandard primes is in one-to-one correspondence with the set of prime
elements in the ring ${}^*\mathbb{Z}$ of nonstandard integers in the sense of Robinson \cite{R}.

\smallskip

Indeed, one may utilize the following construction, which was thoroughly studied\footnote{also cf.
\cite{C}} in \cite{LLS}. Consider the ring $\mathbb{Z}^{\omega}=\prod_{m\in\mathbb{N}}\mathbb{Z}$ -
the product of countably many copies of $\mathbb{Z}$ indexed by $\mathbb{N}$. The minimal prime
ideals of $\mathbb{Z}^{\omega}$ are in bijection with the set of all ultrafilters on $\mathbb{N}$
(perhaps it is opportune to remind that the latter is precisely the Stone-Cech compactification
$\beta\mathbb{N}$ of $\mathbb{N}$ as a discrete space). Explicitly, if for every
$a=(a_m)\in\mathbb{Z}^{\omega}$ one defines the \textbf{support complement} as
\begin{equation*}
\theta(a)=\lbrace m\in\mathbb{N}\;|\;a_m=0\rbrace
\end{equation*}
and for an arbitrary ultrafilter $\mathcal{U}\in 2^{\mathbb{N}}$ sets
\begin{equation*}
(\mathcal{U})=\lbrace a\in\mathbb{Z}^{\omega}\;|\;\theta(a)\in\mathcal{U}\rbrace,
\end{equation*}
then one obtains a minimal prime ideal of $\mathbb{Z}^{\omega}$. It is easily shown that every
minimal prime ideal is of such a form. Of course, the index set $\mathbb{N}$ may be replaced by any
set $I$, after which one easily gets the description of minimal primes of $\mathbb{Z}^I$ (since
those correspond to ultrafilters, there are exactly $2^{2^{|I|}}$ of them if $I$ is infinite and
$|I|$ when $I$ is a finite set). Note that in the case of finite index set all ultrafilters are
principal, and the corresponding $(\mathcal{U})$ are of the form $\mathbb{Z}\times\cdots\times
(0)\times\cdots\times\mathbb{Z}$ - a textbook example.

Similarly, one may replace each copy of $\mathbb{Z}$ by an arbitrary integral domain and repeat the
construction above. If for instance all the rings in the product happen to be fields, then, since
the product of any number of fields is von Neumann regular, the ideal $(\mathcal{U})$ will also be
maximal.
\smallskip

The ring of nonstandard integers  may be viewed as a quotient (an ultrapower)
\begin{equation*}
\mathbb{Z}^{\omega}/(\mathcal{U})={}^*\mathbb{Z}.
\end{equation*}
The class of $\mathcal{U}$-congruent sequences $[p]$ corresponds to an element (also an equivalence
class) in ${}^*\mathbb{Z}$, which may as well as $[p]$ be represented by a prime number sequence
$p=(p_m)$, only in the latter case some but not too many of the primes $p_m$ may be replaced by
arbitrary integers. For all intents and purposes, this difference is insignificant.

\smallskip
Also, observe that $[p]$ indeed generates a maximal prime ideal in ${}^*\mathbb{Z}$: if one for
(any) $p\in [p]$ defines an ideal in $\mathbb{Z}^{\omega}$ as
\begin{equation*}
(p,\; \mathcal{U})=\lbrace a\in\mathbb{Z}^{\omega}\;|\;\lbrace m\;|\; a_m\in p_m\mathbb{Z}\rbrace\in\mathcal{U}\rbrace,
\end{equation*}
then, taking the quotient $\mathbb{Z}^{\omega}/(p,\; \mathcal{U})$ in two different ways, one
arrives at an isomorphism
\begin{equation*}
{}^*\mathbb{Z}/([p])\simeq \left(\prod_{m}\mathbb{Z}_{p_m}\right)/(\mathcal{U}),
\end{equation*}
and the right-hand side is a field by the preceding remark.
\smallskip
For a fixed non-principal $\mathcal{U}$ and an infinite prime $[p]$, we will call the quotient
\begin{equation*}
\mathbb{Z}_{[p]}\equiv {}^*\mathbb{Z}/([p])
\end{equation*}
\textbf{the nonstandard residue field} of $[p]$. Under our assumptions this field has
characteristic zero.

\medskip

\subsection{Algebraic closure of nonstandard residue field}
We have seen that the objects $[p]$ - the infinite prime - behaves similarly to the usual prime
number in the sense that a version of a residue field corresponding to this object may be
constructed. Note that the standard residue fields are contained as a degenerate case in this
construction, namely if we drop the condition of unboundedness and instead consider
$\mathcal{U}$-stationary sequences, we will arrive at a residue field isomorphic to
$\mathbb{Z}_{p}$, with $p$ being the image of the stationary sequence in the chosen class. The
fields of the form $\mathbb{Z}_{[p]}$ are a realization of what is known as pseudofinite field, cf.
\cite{BH}.

\smallskip

The nonstandard case is surely more interesting. While the algebraic closure of a standard residue
field is countable, the nonstandard one itself has the cardinality of the continuum. Its algebraic
closure is also of that cardinality and has characteristic zero, which implies that it is
isomorphic to the field of complex numbers. We proceed by demonstrating these facts.
\smallskip

\begin{prop}
For any infinite prime $[p]$ the residue field $\mathbb{Z}_{[p]}$ has the cardinality of the
continuum\footnote{There is a general statement on cardinality of ultraproduct due to Frayne,
Morel, and Scott \cite{FMS}. We believe the proof of this particular instance may serve as a neat
example of what we are dealing with in the present paper.}.
\end{prop}
\begin{proof}
It suffices to show there is a surjection
\begin{equation*}
h^*:\mathbb{Z}_{[p]}\rightarrow \mathfrak{P},
\end{equation*}
where $\mathfrak{P}=\lbrace 0,1\rbrace^{\omega}$ is the Cantor set given as the set of all
countable strings of bits with the 2-adic metric
\begin{equation*}
d_2(x,y)=1/k,\;\;k=\min\lbrace m\;|\;x_m\neq y_m\rbrace.
\end{equation*}
The map $h^*$ is constructed as follows. If $\mathfrak{Z}\subset \mathfrak{P}$ is the subset of all
strings with finite number of ones in them, and
\begin{equation*}
e:\mathbb{Z}_+\rightarrow\mathfrak{Z},\;\;e\left(\sum_{k<m}f_k2^k\right)=(f_1,\ldots,f_{m-1},0,\ldots)
\end{equation*}
is the bijection that sends a nonnegative integer to its binary decomposition, then for a class
representative $a=(a_m)\in[a]\in\mathbb{Z}_{[p]}$ set $h^*(a)$ to be the (unique) ultralimit of the
sequence of points $\lbrace x_m=e(a_m)\rbrace$. The correctness of this map rests on the property
of the Cantor set being Hausdorff quasi-compact. Surjectivity is then established directly:
consider an arbitrary $x\in\mathfrak{P}$. For each $m\in\mathbb{N}$ the set
\begin{equation*}
\mathfrak{P}_m=\lbrace e(0),e(1),\ldots,e(p_m-1)\rbrace
\end{equation*}
consists of $p_m$ distinct points. Let $x_m$ be the nearest to $x$ point from this set with respect
to the 2-adic metric. The sequence $(p_m)$ is unbounded, so that for every $m\in\mathbb{N}$ the
index subset
\begin{equation*}
A_m=\lbrace k\in\mathbb{N}\;|\;p_k>2^m\rbrace
\end{equation*}
belongs to the ultrafilter $\mathcal{U}$. It is easily seen that for every $k\in A_m$ one has:

\begin{equation*}
d_2(x,x_k)<1/m
\end{equation*}
But that effectively means that the sequence $(x_m)$ has the ultralimit $x$, after which
$a_m=e^{-1}(x_m)$ yields the desired preimage.
\end{proof}
As an immediate corollary of this proposition and the well-known Steinitz theorem, one has
\begin{thm}
The algebraic closure $\overline{\mathbb{Z}_{[p]}}$ of $\mathbb{Z}_{[p]}$ is isomorphic to the
field of complex numbers.
\end{thm}
\smallskip

We now fix the notation for the aforementioned isomorphisms in order to employ it in the next
section.

For any nonstandard prime $[p]\in {}^*\mathbb{P}$ fix an isomorphism
$\alpha_{[p]}:\mathbb{C}\rightarrow \overline{\mathbb{Z}_{[p]}}$ coming from the preceding theorem.
Denote by $\Theta_{[p]}:\overline{\mathbb{Z}_{[p]}}\rightarrow \overline{\mathbb{Z}_{[p]}}$ the
nonstandard Frobenius automorphism - that is, a well-defined field automorphism that sends a
sequence of elements to a sequence of their $p_m$-th powers:
\begin{equation*}
(x_m)\mapsto(x_m^{p_m}).
\end{equation*}
The automorphism $\Theta_{[p]}$ is identical on $\mathbb{Z}_{[p]}$; conjugated by $\alpha_{[p]}$,
it yields a wild automorphism of complex numbers, as by assumption no finite power of it (as
always, in the sense of index subsets $A\in\mathcal{U}$) is the identity homomorphism.

\bigskip

\subsection{Extension of the Weyl algebra}
The $n$-th Weyl algebra $W_{n,\mathbb{C}}\simeq W_{n,\overline{\mathbb{Z}_{[p]}}}$ can be realized
as a proper subalgebra of the following ultraproduct of algebras
\begin{equation*}
\mathcal{W}_n(\mathcal{U},[p])=\left(\prod_{m\in\mathbb{N}}W_{n,\mathbb{F}_{p_m}}\right)/\mathcal{U}.
\end{equation*}
Here for any $m$ the field $\mathbb{F}_{p_m}=\overline{\mathbb{Z}_{p_m}}$ is the algebraic closure
of the residue field $\mathbb{Z}_{p_m}$. This larger algebra contains elements of the form
$(\zeta^{I_m})_{m\in\mathbb{N}}$ with unbounded $|I_m|$ - something which is not present in
$W_{n,\overline{\mathbb{Z}_{[p]}}}$, hence the proper embedding. Note that for the exact same
reason (with degrees $|I_m|$ of differential operators having been replaced by degrees of minimal
polynomials of algebraic elements) the inclusion
\begin{equation*}
\overline{\mathbb{Z}_{[p]}}\subseteq \left(\prod_{m\in\mathbb{N}}{F}_{p_m}\right)/\mathcal{U}
\end{equation*}
is also proper.
\smallskip

It turns out that, unlike its standard counterpart $W_{n,\mathbb{C}}$, the algebra
$\mathcal{W}_n(\mathcal{U},[p])$ has a huge center described in this proposition:
\begin{prop}
The center of the ultraproduct of Weyl algebras over the sequence of algebraically closed fields
$\lbrace \mathbb{F}_{p_m}\rbrace$ coincides with the ultraproduct of centers of
$W_{n,\mathbb{F}_{p_m}}$:
\begin{equation*}
C(\mathcal{W}_n(\mathcal{U},[p]))=\left(\prod_{m}C(W_{n,\mathbb{F}_{p_m}})\right)/\mathcal{U}.
\end{equation*}
\end{prop}
The proof is elementary and is left to the reader. As in positive characteristic the center
$C(W_{n,\mathbb{F}_{p}})$ is given by the polynomial algebra
\begin{equation*}
\mathbb{F}_{p}[x_1^{p},\ldots,x_n^{p},y_1^{p},\ldots,y_n^{p}]\simeq \mathbb{F}_{p}[\xi_1,\ldots,\xi_{2n}],
\end{equation*}
there is an injective $\mathbb{C}$-algebra homomorphism
\begin{equation*}
\mathbb{C}[\xi_1,\ldots\xi_{2n}]\rightarrow \left(\prod_{m}\mathbb{F}_{p_m}[\xi_1^{(m)},\ldots\xi_{2n}^{(m)}]\right)/\mathcal{U}
\end{equation*}
from the algebra of regular functions on $\mathbb{A}_{\mathbb{C}}^{2n}$ to the center of
$\mathcal{W}_n(\mathcal{U},[p])$, evaluated on the generators in a straightforward way:
\begin{equation*}
\xi_i\mapsto[(\xi_i^{(m)})_{m\in\mathbb{N}}].
\end{equation*}
Just as before, this injection is proper.

Furthermore, the image of this monomorphism (the set which we will simply refer to as the
polynomial algebra) may be endowed with the canonical Poisson bracket. Recall that in positive
characteristic case for any $a,b\in \mathbb{Z}_p[\xi_1,\ldots,\xi_{2n}]$ one can define
\begin{equation*}
\lbrace a,b\rbrace=-\pi\left(\frac{[a_0,b_0]}{p}\right).
\end{equation*}
Here $\pi:W_{n,\mathbb{Z}}\rightarrow W_{n,\mathbb{Z}_p}$ is the modulo $p$ reduction of the Weyl
algebra, and $a_0, b_0$ are arbitrary lifts of $a,b$ with respect to $\pi$. The operation is well
defined, takes values in the center and satisfies the Leibnitz rule and the Jacobi identity. On the
generators one has
\begin{equation*}
\lbrace \xi_i,\xi_j\rbrace=\omega_{ij}.
\end{equation*}
The Poisson bracket is trivially extended to the entire center
$\mathbb{F}_{p}[\xi_1,\ldots,\xi_{2n}]$ and then to the ultraproduct of centers. Observe that the
Poisson bracket of two elements of bounded degree is again of bounded degree, hence one has the
bracket on the polynomial algebra.
\medskip

\subsection{Endomorphisms and symplectomorphisms}

The point of this construction lies in the fact that thus defined Poisson structure on the
(injective image of) polynomial algebra is preserved under all endomorphisms of
$\mathcal{A}_n(\mathcal{U},[p])$ of bounded degree. Every endomorphism of the standard Weyl algebra
is specified by an array of coefficients $(a_{i,I})$ (which form the images of the generators in
the standard basis); these coefficients are algebraically dependent, but with only a finite number
of bounded-order constraints. Hence the endomorphism of the standard Weyl algebra can be extended
to the larger algebra $\mathcal{A}_n(\mathcal{U},[p])$. The restriction of any such obtained
endomorphism on the polynomial algebra $\mathbb{C}[\xi_1,\ldots,\xi_{2n}]$ preserves the Poisson
structure. In this setup the automorphisms of the Weyl algebra correspond to symplectomorphisms of
$\mathbb{A}_{\mathbb{C}}^{2n}$.

\medskip

\textbf{Example}. If $x_i$ and $y_i$ are standard generators, then one may perform a linear
symplectic change of variables:
\begin{eqnarray*}
f(x_i)=\sum_{j=1}^na_{ij}x_j+\sum_{j=1}^{n}a_{i,n+j}y_j,\;\;i=1,\ldots,n,\\
f(d_i)=\sum_{j=1}^na_{i+n,j}x_j+\sum_{j=1}^{n}a_{i+n,n+j}y_j,\;\; a_{ij}\in\mathbb{C}.
\end{eqnarray*}
In this case the corresponding polynomial automorphism $f^c$ of
\begin{equation*}
\mathbb{C}[\xi_1,\ldots,\xi_{2n}]\simeq\mathbb{C}[x_1^{[p]},\ldots,x_n^{[p]},y_1^{[p]},\ldots,y_n^{[p]}]
\end{equation*}
acts on the generators $\xi$ as
\begin{equation*}
f^c(\xi_i)=\sum_{j=1}^{2n}(a_{ij})^{[p]}\xi_j,
\end{equation*}
where the notation $(a_{ij})^{[p]}$ means taking the base field automorphism that is conjugate to
the nonstandard Frobenius via the Steinitz isomorphism.

Let $\gamma:\mathbb{C}\rightarrow\mathbb{C}$ be an arbitrary automorphism of the field of complex
numbers. Then, given an automorphism $f$ of the Weyl algebra $W_{n,\mathbb{C}}$ with coordinates
$(a_{i,I})$, one can build another algebra automorphism using the map $\gamma$. Namely, the
coefficients $\gamma(a_{i,I})$ define a new automorphism $\gamma_*(f)$ of the Weyl algebra, which
is of the same degree as the original one. In other words, every automorphism of the base field
induces a map $\gamma_*:W_{n,\mathbb{C}}\rightarrow W_{n,\mathbb{C}}$ which preserves the structure
of the ind-object. It obviously is a group homomorphism.

\medskip

Now, if $P_{n,\mathbb{C}}$ denotes the commutative polynomial algebra with Poisson bracket, we may
define an $\Ind$-group homomorphism $\phi:\Aut(W_{n,\mathbb{C}})\rightarrow \Aut(P_{n\mathbb{C}})$
as follows. Previously we had a morphism $f\mapsto f^c$, however as the example has shown it
explicitly depends on the choice of the infinite prime $[p]$. We may eliminate this dependence by
pushing the whole domain $\Aut(W_{n,\mathbb{C}})$ forward with a specific base field automorphism
$\gamma$, namely $\gamma=\Theta_{[p]}^{-1}$ - the field automorphism which is Steinitz-conjugate
with the inverse nonstandard Frobenius, and only then constructing the symplectomorphism
$f^c_{\Theta}$ as the restriction to the (nonstandard) center. For the subgroup of tame
automorphisms such as linear changes of variables this procedure has a simple meaning: just take
the $[p]$-th root of all coefficients $(a_{i,I})$ first. We thus obtain a group homomorphism which
preserves the filtration by degree and is in fact well-behaved with respect to the Zariski topology
on $\Aut$ (indeed, the filtration $\Aut^{N}\subset \Aut^{N+1}$ is given by Zariski-closed
embeddings). Formally, we have a proposition:

\begin{prop}
There is a system of morphisms
\begin{equation*}
\phi_{[p],N}:\Aut^{\leq N}(W_{n,\mathbb{C}})\rightarrow \Aut^{\leq N}(P_{n,\mathbb{C}}).
\end{equation*}
such that the following diagram commutes for all $N\leq N'$:
\begin{equation*}
\begin{tikzcd}
\Aut^{\leq N}(W_{n,\mathbb{C}}) \arrow{r}{\phi_{[p],N}} \arrow{d}{\mu_{NN'}} & \Aut^{\leq N}(P_{n,\mathbb{C}})\arrow{d}{\nu_{NN'}}\\
\Aut^{\leq N'}(W_{n,\mathbb{C}})\arrow{r}{\phi_{[p],N'}} & \Aut^{\leq N'}(P_{n,\mathbb{C}})
\end{tikzcd}
\end{equation*}
The corresponding direct limit of this system is given by $\phi_{[p]}$, which maps a Weyl algebra
automorphism $f$ to a symplectomorphism $f^c_{\Theta}$.
\end{prop}
\medskip

The strengthened form of the Belov -- Kontsevich conjecture (Conjecture \ref{mainconj}) then
states:
\begin{conj}
$\phi_{[p]}$ is a group isomorphism.
\end{conj}
Injectivity may be established right away.
\begin{thm}
$\phi_{[p]}$ is an injective homomorphism.
\end{thm}
(See \cite{4} for the fairly elementary proof).

\section{$\Phi$ is a morphism of normalized $\Ind$-schemes}
Let us at first assume that the Belov -- Kontsevich conjecture holds, with $\phi_{[p]}$ furnishing
the isomorphism between the automorphism groups. (This would be the case if all automorphisms in
$\Aut(W_{n,\mathbb{C}})$ were tame, which is unknown at the moment for $n>1$. See also \cite{K-BE4}
for the justification of the above assumption.)

Let $[p]$ and $[p']$ be two distinct classes of $\mathcal{U}$-congruent prime number sequences -
that is, two distinct infinite primes. We then have the following diagram:
\begin{equation*}
\begin{tikzcd}
\Aut(W_{n,\mathbb{C}}) \arrow{r}{\phi_{[p]}} \arrow{d}{isom} & \Aut(P_{n,\mathbb{C}})\arrow{d}{isom}\\
\Aut(W_{n,\mathbb{C}})\arrow{r}{\phi_{[p']}} & \Aut(P_{n,\mathbb{C}})
\end{tikzcd}
\end{equation*}
with all arrows being isomorphisms. Vertical isomorphisms answer to different presentations of
$\mathbb{C}$ as $\overline{\mathbb{Z}_{[p]}}$ and $\overline{\mathbb{Z}_{[p']}}$. The corresponding
automorphism $\mathbb{C}\rightarrow \overline{\mathbb{Z}_{[p]}}$ is denoted by $\alpha_{[p]}$ for
any $[p]$.
\smallskip

The fact that all the arrows in the diagram are isomorphisms allows one instead to consider a loop
morphism of the form
\begin{equation*}
\Phi:\Aut(W_{n,\mathbb{C}})\rightarrow \Aut(W_{n,\mathbb{C}}).
\end{equation*}
Furthermore, as it was noted in the previous section, the morphism $\Phi$ belongs to \\
$\Aut(\Aut(W_{n,\mathbb{C}}))$.




In order to prove the Main Theorem, we need to prove that $\Phi$ -- and its augmented analogue
$\Phi^h$ -- is a morphism of (normalized) $\Ind$-schemes. This is achieved by the following
Proposition.
\begin{prop}\label{propgabber}
For an arbitrary Weyl algebra automorphism $f\in\Aut^{\leq N}(W_{n,\mathbb{C}})$, the coordinates
of $\Phi_N(f)$ are given by algebraic functions of the coordinates of $f$.
\end{prop}
\begin{proof}
The statement of this Proposition is contained in Theorem 2 of \cite{4}; for the convenience of the
reader, we reproduce the proof with technical modifications in accordance with the context of this
section. The idea of utilizing local ad-nilpotency is originally due to O. Gabber.

 \smallskip

It suffices to demonstrate the Proposition as a property of $\phi_{[p]}$. More precisely, it is
enough to show that, given an automorphism $f_p$ of the Weyl algebra in positive characteristic $p$
with coordinates $(a_{i,I})$, its restriction to the center (a symplectomorphism) $f^c_p$ has
coordinates which are algebraic in $(a_{i,I}^p)$, with the additional property of being implicitly
defined by polunomial identities of degree independent of $p$ (so that one can return to
characteristic zero).

\smallskip

The switch to positive characteristic and back is performed for a fixed $f\in\Aut
(W_{n,\mathbb{C}})$ on an index subset $A_f\in\mathcal{U}$.

\smallskip

Let $f$ be an automorphism of $W_{n,\mathbb{C}}$ and let $N=\Deg f$ be its degree. The automorphism
$f$ is specified by its coordinates $a_{i,I}\in\mathbb{C}$, $i=1,\ldots,2n$, $I=\lbrace
i_1,\ldots,i_{2n}\rbrace$, obtained from the decomposition of algebra generators $\zeta_i$ in the
standard basis of the free module:
\begin{equation*}
f(\zeta_i)=\sum_{i,I}a_{i,I}\zeta^I,\;\;\zeta^I=\zeta_1^{i_1}\cdots\zeta_{2n}^{i_{2n}}.
\end{equation*}

Let $(a_{i,I,p})$ denote the class $\alpha_{[p]}(a_{i,I})$, $p=(p_m)$, and let $\lbrace
R_k(a_{i,I}\;|\;i,I)=0\rbrace_{k=1,\ldots,M}$ be a finite set of algebraic constraints for
coefficients $a_{i,I}$. Let us denote by $A_1,\ldots,A_M$ the index subsets from the ultrafilter
$\mathcal{U}$, such that $A_k$ is precisely the subset on whose indices the constraint $R_k$ is
valid for $(a_{i,I,p})$. Take $A_f=A_1\cap\ldots\cap A_M\in\mathcal{U}$ and for $p_m$, $m\in A_f$,
define an automorphism $f_{p_m}$ of the Weyl algebra in positive characteristic
$W_{n,\mathbb{F}_{p_m}}$ by setting
\begin{equation*}
f_{p_m}(\zeta_i)=\sum_{i,I}a_{i,I,p_m}\zeta^I.
\end{equation*}
All of the constraints are valid on $A_f$, so that $f$ corresponds to a class $[f_p]$ modulo
ultrafilter $\mathcal{U}$ of automorphisms in positive characteristic. The degree of every
$f_{p_m}$ ($m\in A_f$) is obviously less than or equal to $N=\Deg f$.


Now consider $f\in\Aut^{\leq N}(W_{n,\mathbb{C}})$ with the index subset $A_f$ over which its
defining constraints are valid. The automorphisms $f_{p_m}=f_p:W_{n,\mathbb{F}_p}\rightarrow
W_{n,\mathbb{F}_p}$ defined for $m\in A_f\in\mathcal{U}$ provide arrays of coordinates $a_{i,I,p}$.
Let us fix any valid $p_m=p$ and denote by $\mathit{F}_{p^k}$ a finite subfield of $\mathbb{F}_p$
which contains the respective coordinates $a_{i,I,p}$ (one may take $k$ to be equal to the maximum
degree of all minimal polynomials of elements $a_{i,I,p}$ which are algebraic over $\mathbb{Z}_p$).

\smallskip

Let $a_1,\ldots,a_s$ be the transcendence basis of the set of coordinates $a_{i,I,p}$ and let
$t_1,\ldots,t_s$ denote $s$ independent (commuting) variables. Consider the field of rational
functions:
\begin{equation*}
\mathit{F}_{p^k}(t_1,\ldots,t_s).
\end{equation*}
\medskip
The vector space
\begin{equation*}
\Der_{\mathbb{Z}_p}(\mathit{F}_{p^k}(t_1,\ldots,t_s),\mathit{F}_{p^k}(t_1,\ldots,t_s))
\end{equation*}
of all $\mathbb{Z}_p$-linear derivations of the field $\mathit{F}_{p^k}(t_1,\ldots,t_s)$ is
finite-dimensional with \\$\mathbb{Z}_p$-dimension equal to $ks$; a basis of this vector space is
given by elements $$\lbrace e_a\mathit{D}_{t_b}\;|\;a=1,\ldots, k,\;b=1,\ldots,s\rbrace$$ where
$e_a$ are basis vectors of the $\mathbb{Z}_p$-vector space  $\mathit{F}_{p^k}$, and
$\mathit{D}_{t_b}$ is the partial derivative with respect to the variable $t_b$.

\smallskip

Set $a_1,\ldots,a_s=t_1,\ldots,t_s$ (i.e. consider an $s$-parametric family of automorphisms), so
that the rest of the coefficients $a_{i,I,p}$ are algebraic functions of $s$ variables
$t_1,\ldots,t_s$. We need to show that the coordinates of the corresponding symplectomorphism
$f_p^c$ are annihilated by all derivations $e_a\mathit{D}_{t_b}$.

Let $\delta$ denote a derivation of the Weyl algebra induced by an arbitrary basis derivation
$e_a\mathit{D}_{t_b}$ of the coefficient field. For a given $i$, let us introduce the short-hand
notation
\begin{equation*}
a=f_p(\zeta_i),\;\;b=\delta(a).
\end{equation*}
\medskip
We need to prove that
\begin{equation*}
\delta(f^c(\xi_i))=\delta(f_p(\zeta_i^p))=0.
\end{equation*}
In our notation $\delta(f_p(\zeta_i^p))=\delta(a^p)$, so by Leibnitz rule we have:
\begin{equation*}
\delta(f_p(\zeta_i^p))=ba^{p-1}+aba^{p-2}+\cdots+a^{p-1}b.
\end{equation*}

Let $\ad_x:W_{n,\mathbb{F}_p}\rightarrow W_{n,\mathbb{F}_p}$ denote a $\mathbb{Z}_p$-derivation of
the Weyl algebra corresponding to the adjoint action (recall that all Weyl algebra derivations are
inner):
\begin{equation*}
\ad_x(y)=[x,y].
\end{equation*}
We will call an element $x\in W_{n,\mathbb{F}_p}$ \textbf{locally ad-nilpotent} if for any $y\in
W_{n,\mathbb{F}_p}$ there is an integer $D=D(y)$ such that
\begin{equation*}
\ad_x^D(y)=0.
\end{equation*}
All algebra generators $\zeta_i$ are locally ad-nilpotent. Indeed, one could take $D(y)=\Deg y + 1$
for every $\zeta_i$.

If $f$ is an automorphism of the Weyl algebra, then $f(\zeta_i)$ is also a locally ad-nilpotent
element for all $i=1,\ldots,2n$. That means that for any $i=1,\ldots,2n$ there is an integer $D\geq
N+1$ such that
\begin{equation*}
\ad_{f_p(\zeta_i)}^D(\delta(f_p(\zeta_i)))=\ad_{a}^D(b)=0.
\end{equation*}

\smallskip

Now, for $p\geq D+1$ the previous expression may be rewritten as
\begin{equation*}
0=\ad_{a}^{p-1}(b)=\sum_{l=0}^{p-1}(-1)^l\binom{p-1}{l}a^lba^{p-1-l}\equiv\sum_{l=0}^{p-1}a^lba^{p-1-l}\;\;(\text{mod\;} p),
\end{equation*}
and this is exactly what we wanted.

\smallskip

We have thus demonstrated that for an arbitrary automorphism $f_p$ of the Weyl algebra in
characteristic $p$ the coordinates of the corresponding symplectomorphism $f_p^c$ are algebraic
functions in the $p$-th powers of the coordinates of $f_p$, provided that $p$ is greater than $\Deg
f_p + 1$.

As the sequence $(\Deg f_{p_m})$ is bounded from above by $N$ for all $m\in A_f$, we see that there
is an index subset $A^*_f\in\mathcal{U}$ such that the coordinates of the symplectomorphism
$f_{p_m}^c$ for $m\in A^*_f$ are algebraic functions in $p_m$-th powers of $a_{i,I,p_m}$.
Furthermore, the degree of these functions' defining polynomials (i.e. polynomials whose roots are
these algebraic functions) is also bounded from above, as $f_p$ is in $\Aut^{\leq N}$. Therefore,
the image $\Phi_N(f)$ of the characteristic zero automorphism will have coordinates algebraic in
the coefficients of $f$.

\end{proof}

\medskip

\begin{remark}
The nature of local ad-nilponency is not to be considered trivial. This is especially so in the
context of ultrafilter decomposition (modulo infinite prime reduction), for various propositions
involving the Weyl algebra in positive characteristic may well contain elements whose degree is
dependent on the prime $p$. This problem presents certain difficulties, sometimes of quite
substantial nature (one notable example being the Dixmier conjecture itself, or rather the
impossibility of its negation by a naive example in positive characteristic with subsequent
ultrafilter reconstruction).

For a more immediate illustration, consider the following elementary automorphism of
$W_{1,\mathbb{F}_p}$
$$
\varphi:\;x\mapsto x+ay^{p-1},\;\;y\mapsto y
$$
with $a$ constant. The rank of this automorphism grows as $p$, which prevents it from being
reconstructed in characteristic zero. Moreover, the problem becomes even more involved if one looks
at its image under the Kontsevich homomorphism: the action upon the central element $x^p$ will
result in the appearance of coefficients which will be (even after the Frobenius untwisting)
explicitly dependent on $p$. Also, thanks to the Weyl algebra commutation relations, the term
$ay^{p-1}x^{p-1}$ in the image of $x^p$ will contribute (in the standard ordering of generators) a
constant term $a$, so that the image looks as this:
$$
\varphi(x^p)=x^p+a^py^{p(p-1)}+a.
$$

This would have a prohibitive effect on the theorems on symplectomorphism approximation (as
developed in \cite{KGE}), were the latter to be applied naively.
\end{remark}

\medskip

Proposition \ref{propgabber} has the following important corollary.

\begin{cor}\label{corgabber}
Let $(\Aut^{\leq N}(W_{n,\mathbb{C}}))^{\nu}$ denote the normalization of the scheme \\
$\Aut^{\leq N}(W_{n,\mathbb{C}})$ (corresponding to the variety defined previously). Then the maps
$\Phi_N$ induce monomorphisms of normal schemes
$$
\Phi_N: (\Aut^{\leq N}(W_{n,\mathbb{C}}))^{\nu}\rightarrow (\Aut^{\leq N}(W_{n,\mathbb{C}}))^{\nu}.
$$
\end{cor}
\begin{proof}
This corollary is immediate since the coefficients of $\Phi_N(f)$ are algebraic in coordinates of
$f$ and are therefore in the coordinate ring of the normalized variety. Injectivity follows from
the assumption on $\phi_{[p]}$ at the beginning of the section.
\end{proof}

In view of Corollary \ref{corgabber}, it is sensible to amend the definitions of automorphism
$\Ind$-schemes we have given in the previous sections. From this point forward, by $\Aut^{\leq N}$
we mean normalized versions of varieties defined in the introduction (thus dropping the superscript
$\nu$ for the sake of notational convenience). The corresponding direct limits will therefore be
normal $\Ind$-schemes. Accordingly, the $\Ind$-map $\Phi$ (corresponding to the group homomorphism
$\Phi$) will always mean the system of morphisms of Corollary \ref{corgabber}.

\begin{remark}
It is possible to establish a stronger version of Proposition \ref{propgabber} (or rather, the
property of $\phi_{[p]}$ demonstrated in the proof) for the subschemes $\TAut$ of tame
automorphisms: the map $\phi_{[p]}$ is polynomial and indeed is an isomorphism (independent of the
choice of infinite prime, as proved in \cite{4}).
\end{remark}

\smallskip

\begin{remark}\label{remextended}
Once the fact that $\Phi$ is a morphism of normalized $\Ind$-schemes has been established, one may
investigate its properties in more detail. In particular (as in the extended version of this
section in \cite{K-BEadd}) one may demonstrate that $\Phi$ preserves asymptotic expansions of
curves in $\Aut$ -- a strong restriction on $\Phi$ from which one can, after some preparation,
derive the fact that $\Phi = \lbrace \Phi_N\rbrace$ is locally unipotent: for each $N$ there is a
power $k = k(N)$ such that
$$
\Phi_N^k =\Phi_N\circ\ldots\circ\Phi_N= \Id.
$$
The proof requires a rather lengthy detour from our main objective and is not reproduced here (we
will address the issue in the arXiv version of this paper, \cite{K-BEadd}).

The proof (of local unipotency and to a certain extent of asymptotic expansion preservation) bears
some resemblance -- at least in general terms -- to the singularity trick which is our method of
attack on the Main Theorem. Also, while the techniques developed in \cite{K-BEadd} seem to be
insufficient for the characterisation of all $\Ind$-automorphisms of $\Aut P_{n,\mathbb{C}}$, in
analogy with the results of \cite{KBYu}, the local unipotency of asymptotic expansion-preserving
$\Ind$-morphisms can be viewed as a property of some relevance to that program.
\end{remark}


The next section is dedicated to the proof of the Main Theorem.

\section{Independence of infinite prime: $\Phi$ in the augmented case}

\subsection{Augmented algebras and the loop morphism}

\subsubsection{Augmented and skew augmented algebras}
In order to conclude the proof of the Main Theorem, we proceed along the way of a certain
topological argument. To do so, we modify our setting, passing from the initial Weyl algebra $W_{n,
\mathbb{C}}$ to another associative unital algebra $W^h_{n, \mathbb{C}}$, which is defined as the
quotient of the free algebra on $(2n+1)$ indeterminates $\mathbb{C}\langle a_1,\ldots, a_n,
b_1,\ldots, b_n, c\rangle$ by the two-sided ideal generated by elements
$$
a_ia_j-a_ja_i,\;\;b_ib_j-b_jb_i\;\;,
b_ia_j-a_jb_i-\delta_{ij}c,\;\;
a_ic-ca_i,\;\;
b_ic-cb_i.
$$
We will denote the images of $a_i$, $b_i$ and $c$ in the quotient algebra by $x_i$, $d_i$ and $h$,
respectively. The generators of $W^h_n$ are subject to the commutation relations which, at least
with respect to the first two sets of generators ($x_i$ and $d_i$), are similar in structure to
those of the Weyl algebra $W_n$, however the new generator $h$ is also involved. This construction
mirrors the familiar commutation relations
$$
[d_i, x_j]=\hbar \delta_{ij}
$$
arising in the setting of deformation quantization.

\smallskip

Accordingly, the commutative Poisson algebra $P_{n,\mathbb{C}}$ is similarly modified, by adding
$h$ and setting $\lbrace p_i, x_j\rbrace = h\delta_{ij}$. This is the classical counterpart of
$W^h_n$, which we denote by $P^h_n$.

\smallskip
The idea of passing to augmented algebras plays a key role in the proof of Conjecture
\ref{mainconj} in our most recent paper \cite{K-BE4}.

\smallskip
It can be verified that these new algebras behave in a way almost identical to the one we described
in the previous sections; in particular, the notions of tame automorphism, tame (modified)
symplectomorphism and homomorphisms
\begin{equation*}
\phi_{[p],N}^{h}:\Aut^{\leq N}(W^h_{n,\mathbb{C}})\rightarrow \Aut^{\leq N}(P^h_{n,\mathbb{C}}).
\end{equation*}
which is identical on the tame points, are present. Note also that in the algebra
$W^h_{n,\mathbb{C}}$, the variable $h$ is central.

\smallskip

The loop morphism

$$
\Phi^h:\Aut(W^h_{n,\mathbb{C}})\rightarrow \Aut(W^h_{n,\mathbb{C}})
$$
(or rather, its normalized version) is then constructed, and a similar problem of verifying it to
be the identity map is posed. It is somewhat easier to deal with the Poisson algebra; the Poisson
counterpart of $\Phi^h$ is the $\Ind$-morphism
$$
\Phi_s^h:\Aut(P^h_{n,\mathbb{C}})\rightarrow \Aut(P^h_{n,\mathbb{C}})
$$
which can be obtained from $\Phi^h$ by interchanging the infinite primes $[p]$ and $[p']$ defining
it.

Our main objective is the proof of the following Proposition.

\begin{prop}\label{mainprop}
$\Phi_s^h$ is the identity map.
\end{prop}

In order to prove Proposition \ref{mainprop}, we will have to distort the augmentation further by
introducing a pair of auxiliary (or skew, as we at times call them) algebras
$W_{n,\mathbb{C}}^h[k_{ij}]$ and $P_{n,\mathbb{C}}^h[k_{ij}]$ (which correspond to augmented Weyl
and Poisson algebras, respectively).

These are defined as follows. Let the augmented Poisson generators be denoted by $\xi_i$ with
$1\leq i\leq 2n$, which we will call the main generators, (the passage from $x_i$ and $p_j$ to
$\xi_i$ is made for the sake of uniformity of notation), and let $[k_{ij}]$ be a skew-symmetric
array (a skew matrix) of central variables. The algebra $P_{n,\mathbb{C}}^h[k_{ij}]$ is generated
by $2n$ commuting variables $\xi_i$, the augmentation variable $h$ and the variables $[k_{ij}]$
(thus being the polynomial algebra in these variables); the Poisson bracket is defined on the
generators $\xi_i$:
$$
\lbrace \xi_i,\xi_j\rbrace = hk_{ij}.
$$
The bracket of any element with $h$ or with any of the $k_{ij}$ is zero. \footnote{Thus, the
augmented Poisson algebra $P^h_{n,\mathbb{C}}$ results from the new algebra
$P_{n,\mathbb{C}}^h[k_{ij}]$ by specialization of the matrix $[k_{ij}]$ to the standard symplectic
form.}

The skew version of the algebra $W_{n,\mathbb{C}}$ is defined analogously.

\smallskip

It is easily seen that the new algebras essentially share the positive-characteristic properties
with $W_n$ and $P_n$, from which it follows that a mapping
$$
\phi^{hk}_{[p]}:\Aut W_{n,\mathbb{C}}^h[k_{ij}]\rightarrow \Aut P_{n,\mathbb{C}}^h[k_{ij}]
$$
analogous to $\phi_{[p]}$ and $\phi^h_{[p]}$ can be defined for every infinite prime $[p]$. In a
manner identical to Section 3 it can be established that this mapping consists of a system of
morphisms of the normalized varieties $\Aut^{\leq N}$.

The grading according to the total degree is specified by $$\Deg h = 0,$$ $$\Deg k_{ij} = 2,$$
$$\Deg \xi_i = 1,$$ although since $h$ and $k_{ij}$ appear as products in the commutation
relations, one could assign degree two to the deformation parameter $h$ and degree zero to the
skew-form variables $k_{ij}$, in analogy with the case of augmented algebra $P^h_n$, while
essentially preserving the $\Ind$-scheme structure of $\Aut$.

\smallskip

We then introduce the (augmented) power series topology by the metric
$$
\rho(\varphi, \psi) = \exp(-\Ht(\varphi-\psi))
$$
where
$$
\varphi-\psi = (\varphi(\xi_1)-\psi(\xi_1),\ldots,\varphi(\xi_{2n})-\psi(\xi_{2n}),\ldots)
$$
is the algebra endomorphism defined by its images (on $\xi_i$, $h$ and $k_{ij}$), and the
\textbf{height} $\Ht(\varphi)$ of an endomorphism is defined as the minimal total degree $m$ such
that in one of the generator images under $\varphi$ a non-zero homogeneous component of degree $m$
exists.

Symbolically, we say that the power series topology is defined via the augmentation ideal $I$
$$
I=(\xi_1,\ldots, \xi_{2n}, h, \lbrace k_{ij}\rbrace)
$$
just as it is so in the commutative case, when every variable carries degree one.

The system of neighborhoods $\lbrace H_N\rbrace$ of the identity automorphism in $\Aut
P_{n,\mathbb{C}}^h[k_{ij}]$ is defined by setting
$$
H_N = \lbrace g\in \Aut P_{n,\mathbb{C}}^h[k_{ij}]\;:\; g(\eta) \equiv \eta \;(\text{mod}\;I^N)\rbrace
$$
(here $\eta$ denotes any generator in the set $\lbrace \xi_1,\ldots, \xi_{2n},h, k_{ij}\rbrace$, so
that elements of $H_N$ are precisely those automorphisms which are identity modulo terms which lie
in $I^N$; again, the phrase "mod $I^N$ " is short-hand for the distance as defined above).


Similar notions of grading, topology, and system of standard neighborhoods of a point, are valid
for the algebra $W_{n,\mathbb{C}}^h[k_{ij}]$ (once the proper ordering of the generators in the
chosen set is fixed).

\smallskip

The point of introducing the skew algebras $W_{n,\mathbb{C}}^h[k_{ij}]$ and
$P_{n,\mathbb{C}}^h[k_{ij}]$ is that a certain singularity analysis procedure (the singularity
tricks mentioned in the introduction) can be implemented for these algebras in full analogy with
the case of the commutative polynomial algebra processed in our preceding study \cite{KBYu}, while
on the other hand there seems to be no straightforward way to demonstrate the singularity trick for
the algebras $W^h_n$ and $P^h_n$.

\smallskip

The algebra $P^h_{n,\mathbb{C}}$ can be connected to its skew counterpart
$P_{n,\mathbb{C}}^h[k_{ij}]$ via localization of the variables $k_{ij}$. The argument is as
follows. One may extend the coefficient ring (which is $\mathbb{C}$) by adjoining the entries
$k'_{lm}$ of the inverse matrix $[k_{ij}]^{-1}$ (together with $k_{ij}$ themselves). In the new
algebra, $P_{n,\mathbb{C}[k_{ij}, k'_{lm}]}^h[k_{ij}]$, a $\mathbb{C}[k_{ij}, k'_{lm}]$-linear
transformation of the generators $\xi_i$ yields a new set of generators $x_i$ and $y_i$
($i=1,\ldots, n$), whose commutation relations are given by
$$
\lbrace y_i, x_i\rbrace = h.
$$
In other words, the localized algebra $P_{n,\mathbb{C}[k_{ij}, k'_{lm}]}^h[k_{ij}]$ is isomorphic
to the algebra $P^h_{n,\mathbb{C}[k_{ij}, k'_{lm}]}$ which is just the extension by  $k_{ij},
k'_{lm}$ of the augmented Poisson algebra $P^h_{n,\mathbb{C}}$.

Given an endomorphism $\varphi$ of the algebra $P_{n,\mathbb{C}}^h[k_{ij}]$ such that
$\varphi(k_{ij})$ is a $\mathbb{C}$-linear combination of the variables $k_{ij}$, its image under
the localization is constructed in an obvious fashion. As the localized skew algebra carries the
structure of the (augmented) Poisson algebra, one may formulate the skew analog of Conjecture
\ref{mainconj} for the subgroups of automorphisms of the skew algebras which act linearly on
$[k_{ij}]$. In fact, it is due to the possibility of the singularity trick that one can prove this
conjecture. We therefore state this result as a theorem.
\begin{thm}\label{mainconjskew}
Let $\Aut_{k}P_{n,\mathbb{C}}^h[k_{ij}]$ and $\Aut_{k}W_{n,\mathbb{C}}^h[k_{ij}]$ denote the
automorphism subgroups of the skew Poisson and Weyl algebras consisting of those automorphisms that
map $k_{ij}$ to $\mathbb{C}$-linear combinations of $k_{ij}$. Then the mapping
$$
\phi^{hk}_{[p]}:\Aut_{k}W_{n,\mathbb{C}}^h[k_{ij}]\rightarrow \Aut_{k}P_{n,\mathbb{C}}^h[k_{ij}]
$$
defined in full analogy\footnote{Thus, the isomorphism is the restriction of $\phi^{hk}_{[p]}$
defined above; it is easy to see that the $\phi^{hk}_{[p]}$ maps elements of
$\Aut_{k}W_{n,\mathbb{C}}^h[k_{ij}]$ to elements of $\Aut_{k}P_{n,\mathbb{C}}^h[k_{ij}]$.} with the
homomorphism $\phi_{[p]}$ of non-augmented algebras is an isomorphism.
\end{thm}

Given this result, one may define, for any two infinite primes $[p]$ and $[p']$, the loop morphisms
$\Phi^{hk}$ (skew Weyl) and $\Phi_s^{hk}$ (skew Poisson). In the broadest terms, the proof of the
Main Theorem on the independence of infinite prime -- and, therefore, the proof of Proposition
\ref{mainprop} -- reduces, thanks to the connection of the skew algebras with the Weyl and Poisson
algebras and the specialization argument of the next subsection, to the verification of the
statement that these morphisms are always the identity morphism. The latter result, in turn, will
be demonstrated to be the immediate consequence of the continuity of the loop morphism in the power
series topology, together with the tame approximation (over the polynomial ring $\mathbb{C}[\lbrace
k_{ij}\rbrace]$) in the localized $h$-augmented Poisson subalgebra and the fact that the loop
morphism is point-wise stable on tame elements.

\smallskip

In order to arrive at the objective, we set up the singularity trick as follows. Consider the
algebra $P_{n+1,\mathbb{C}}^h[k_{ij}]$ with $(2n+2)$ main generators $\lbrace \xi_1,\ldots,
\xi_{2n}, u,v\rbrace$. Let
$$
\Aut_{u,v,k} P_{n+1,\mathbb{C}}^h[k_{ij}]
$$
denote the set of all automorphisms $\varphi$ of $P_{n+1,\mathbb{C}}^h[k_{ij}]$ such that:

1. $\varphi(\xi_i) = \xi_i + S_i$, where $S_i$ is a polynomial (in $\xi_i$, $u$, $v$, $h$ and
$k_{ij}$) such that its height with respect to $\lbrace \xi_1,\ldots, \xi_{2n}, u,v\rbrace$ is at
least two.

2. $\varphi(u) = u$, $\varphi(v) = v$.

3. $\varphi(k_{ij})$ is a $\mathbb{C}$-linear combination of $k_{ij}$, i.e.
$\varphi\in\Aut_{k}P_{n+1,\mathbb{C}}^h[k_{ij}]$.

Define the grading as in the prequel: $\xi_i$, $u$, $v$ carry degree one, $h$ carries degree zero,
and $k_{ij}$ carry degree two.

Denote by $H_N^{u,v,k}$ the subgroups of $\Aut_{u,v,k} P_{n+1,\mathbb{C}}^h[k_{ij}]$ consisting of
elements which are the identity map modulo terms of height $N$ with respect to the grading defined
above.

\smallskip

The purpose of the singularity trick set up below is the proof of the following result, which
establishes continuity of the loop morphism in the power series topology.

\begin{prop}\label{skewthetaprop}
If $\Phi_s^{hk}$ is the morphism defined as above for any pair of infinite primes $[p]$, $[p']$,
then
$$
\Phi_s^{hk}(H_N^{u,v,k}) = H_N^{u,v,k}.
$$
\end{prop}

The singularity trick is essentially a criterion for an automorphism $\varphi$ to be an element of
$H_N^{u,v,k}$, expressed in terms of asymptotic behavior of certain parametric families associated
to it. The parametric families of automorphisms are constructed from $\varphi$ by conjugating it
with $\mathbb{C}$-linear changes of the main generators (the latter are given by the set $\lbrace
\xi_1,\ldots, \xi_{2n}\rbrace$). Such parameterized variable changes are given by $(2n+2)$ by
$(2n+2)$ matrices $\Lambda(t)$ with
$$
(\xi_1,\ldots, \xi_{2n},u,v)\mapsto (\xi_1,\ldots, \xi_{2n},u,v)\Lambda(t)
$$
representing the action (such transformations of the main generators induce appropriate mappings of
$[k_{ij}]$). Note that if $\varphi$ is in $H_N^{u,v,k}$, then the conjugation by $\Lambda(t)$ is
also in $H_N^{u,v,k}$, as the action upon $u$ and $v$ is that of $\Lambda(t)\circ \Lambda(t)^{-1}$.

\smallskip

We are going to examine the behavior of such one-parameter families near singularities of
$\Lambda(t)$.

Suppose that, as $t$ tends to zero, the $i$-th eigenvalue of $\Lambda(t)$ also tends to zero as $t^{m_i}$, $m_i\in\mathbb{N}$. 

Let $\lbrace m_i,\;i=1,\ldots 2n+2\rbrace$ be the set of degrees of singularity of eigenvalues of
$\Lambda(t)$ at zero. Suppose that for every pair $(i,j)$ the following holds: if $m_i\neq m_j$,
then there exists a positive integer $M$ such that
$$
\text{either\;\;} m_iM\leq m_j\;\;\text{or\;\;}m_jM\leq m_i.
$$
We will call the largest such $M$ the \textbf{order} of $\Lambda(t)$ at $t=0$. As $m_i$ are all set
to be positive integer, the order equals the integer part of
$\frac{m_{\text{max}}}{m_{\text{min}}}$.

We now formulate the criterion
\begin{prop}[Singularity trick]\label{singtrick}
An element $\varphi\in\Aut_{u,v,k} P_{n+1,\mathbb{C}}^h[k_{ij}]$ belongs to $H_N^{u,v,k}$ if and
only if for every linear matrix curve $\Lambda(t)$ of order $\leq N$ the curve
$$
\Lambda(t)\circ\varphi\circ\Lambda(t)^{-1}
$$
does not have a singularity (a pole) at $t=0$.
\end{prop}
\begin{proof}
Suppose $\varphi\in H_N^{u,v,k}$ and fix a one-parametric family $\Lambda(t)$. Without loss of
generality, we may assume that the first $2n$ main generators $\lbrace \xi_1,\ldots,
\xi_{2n}\rbrace$ correspond to eigenvectors of $\Lambda(t)$. If $\xi_i$ denotes any of these main
generators, then the action of $\Lambda(t)\circ\varphi\circ \Lambda(t)^{-1}$ upon it reads
$$
\Lambda(t)\circ\varphi\circ \Lambda(t)^{-1}(\xi_i) = \xi_i + t^{-m_i}\sum_{l_1+\cdots+l_{2n} = N}a_{l_1\ldots l_{2n}}t^{m_1l_1+\cdots+m_{2n}l_{2n}}P_{i}(\xi_1,\ldots, \xi_{2n},h,k_{ij}) + S_i
$$
where $P_i$ is homogeneous of total degree $N$ (in the previously defined grading) and the height
of $S_i$ is greater than $N$. One sees that for any choice of $l_1,\ldots,l_{2n}$ in the sum, the
expression
$$
m_1l_1+\cdots+m_{2n}l_{2n} - m_i\geq m_{\text{min}}\sum l_j-m_i=m_{\text{min}}N-m_i\geq 0,
$$
so whenever $t$ goes to zero, the coefficient will not go to infinity. The same argument applies to
higher-degree monomials within $S_i$.

The other direction is established by contraposition. Assuming $\varphi\notin H_N^{u,v,k}$, we need
to prove the existence of linear curves with suitable eigenvalue behavior near $t=0$ which create
singularities via conjugation with the given automorphism.

Suppose first that the image of $\xi_1$ under $\varphi$ possesses a monomial which is not divisible
by $\xi_1$ or any $k_{1j}$ ($j\neq 1$). Then one can take $m_1$ and $m_2<m_1$ such that
$$
(N+1)m_2\geq m_1\geq Nm_2
$$
and set the curve $\Lambda(t)$ to be given by a diagonal matrix with entries
$t^{m_1},\;t^{m_2},\;t^{m_2},\ldots$. It is easily checked that conjugation of $\varphi$ by this
curve creates a pole at the coefficient of the chosen monomial.

The general case can be reduced to this special case by means of transformations of the form
($\lambda$ and $\delta$ are suitable constants)
\begin{gather*}
\xi_1\mapsto \xi_1 + \lambda u + \delta v,\\
k_{ij}\mapsto k_{ij},\;\;1<i,j\leq 2n,\\
k_{1j}\mapsto k_{1j} + \lambda k_{2n+1,j} + \delta k_{2n+2,j},\\
k_{1,2n+1}\mapsto k_{1,2n+1} + \delta k_{2n+2,2n+1},\\
k_{1,2n+2}\mapsto k_{1,2n+2} + \lambda k_{2n+1,2n+2}.
\end{gather*}
Conjugation with these transformations create in the image of $\xi_1$ under the resulting
automorphism a monomial from the previous case. In order to obtain the curve $\Lambda(t)$ from the
diagonal curve acting on the conjugated automorphism, one needs only conjugate it with the inverse
of the above transform. The singularity trick is proved.

\end{proof}

\begin{remark}\label{remarkstable}
This version of the singularity trick differs slightly from the original result (Lemma 3.6 of
\cite{KBYu}) -- namely, we have fixed a pair of main variables and consider only subgroups of
automorphisms stable on these variables. This is needed to eliminate from our consideration those
automorphisms whose linear part is a homothety (the proof of this general case seems to be
considerably more complicated) as well as to guarantee the existence of the transformations in the
proof above. Furthermore, if we were to state the singularity trick in its original form as in
\cite{KBYu}, we would have to resolve the following issue: there exist polynomial endomorphisms
whose higher-degree components are arranged in such a way that our method of conjugation with
linear variable changes would not produce the monomial we sought due to cancelling out of terms.
These pathological cases are never automorphisms, however, as one can see by evaluating the
Jacobian, which will be in these cases a non-trivial polynomial. In our case, the necessity of
processing such pathologies is negated by the introduction of the stable variables and the
existence of elementary transformations as above.

Conceptually the stable case is similar to the usual case of algebra over $2n$ main generators,
although in the proof of the main proposition we will have to separate the $(u,v)$-plane from the
rest of the variables, thus slightly complicating the extension of the coefficient ring.
\end{remark}

The implementation of the singularity trick in the proof of Proposition \ref{skewthetaprop}
requires the following general fact.
\begin{lem} \label{lem1} Let
$$
\Phi: X\rightarrow Y
$$
be a morphism of affine algebraic sets, and let $\varphi(t)$ be a curve (more simply, a
one-parameter family of points) in $X$. Suppose that $\varphi(t)$ does not tend to infinity as
$t\rightarrow 0$. Then the image $\Phi \varphi(t)$ under $\Phi$ also does not tend to infinity as
$t\rightarrow 0$.
\end{lem}
The proof of the Lemma is an easy exercise and is left to the reader.

Proposition \ref{skewthetaprop} is now an elementary consequence of the above Lemma together with
the singularity trick. As $\Phi_s^{hk}$ is bijective thanks to Theorem \ref{mainconjskew}, it
suffices to demonstrate that
$$
\Phi_s^{hk}(H_N^{u,v,k}) \subseteq H_N^{u,v,k}
$$
for all $N$.

Let us assume the contrary -- i.e. that for some $N$
$$
\Phi_s^{hk}(H_N^{u,v,k}) \nsubseteq H_N^{u,v,k}.
$$
Then there exists an element $\varphi\in H_N^{u,v,k}$ such that its image
$\Phi_s^{hk}(\varphi)\notin H_N^{u,v,k}$. By Proposition \ref{singtrick}, there is a linear
automorphism (matrix) curve $\Lambda(t)$ of order $\leq N$ such that the curve
$$
\Lambda(t)\circ \Phi_s^{hk}(\varphi)\circ \Lambda(t)^{-1}
$$
has a pole at $t=0$. Since $\Phi_s^{hk}$ is point-wise stable on linear variable changes, the
latter curve is the image under $\Phi_s^{hk}$ of the curve
$$
\Lambda(t)\circ \varphi\circ \Lambda(t)^{-1}.
$$
By our assumption, $\varphi\in H_N^{u,v,k}$; therefore, by Proposition \ref{singtrick}, the curve
above has no singularity at $t=0$. But then the statement that the curve $$\Lambda(t)\circ
\Phi_s^{hk}(\varphi)\circ \Lambda(t)^{-1}$$ -- which is the image of the former curve under the
morphism $\Phi_s^{hk}$ -- has a singularity at $t=0$ yields a contradiction with Lemma \ref{lem1}.
Proposition \ref{skewthetaprop} is proved.


\subsubsection{The $h$-augmented loop morphism is the identity map}

We now apply the developed theory to the proof of the main Proposition \ref{mainprop}. The plan is
as follows. Starting with an arbitrary automorphism $\varphi$ of the augmented Poisson algebra (we
need to show that $\Phi_s^h(\varphi) = \varphi$), we realize it as an automorphism of an algebra
(with the added stable variables $u$ and $v$) over a suitable polynomial ring; that algebra will be
isomorphic to an appropriate localization -- in the sense outlined above, i.e. an extension of the
coefficient ring by addition of entries of the matrix $[k_{ij}]^{-1}$ as well as elements
$k_{ij}^{-1}$ -- of the skew Poisson algebra. We will conjugate this automorphism of the
localization by an appropriate mapping (on which the loop morphism is stable) in order to get rid
of the negative powers of $k_{ij}$; the resulting map will therefore correspond to an automorphism
of the skew Poisson $\mathbb{C}$-algebra. As an automorphism of the $h$-augmented Poisson algebra
over the ring $\mathbb{C}[\lbrace hk_{ij}\rbrace]$, it admits a tame sequence converging to it: the
proof is a straightforward modification of the argument of \cite{KGE}. The sequence and its limit
will thus be realized as a convergent sequence of automorphisms of the localized skew algebra, at
which point another conjugation will be required to dispose of the denominators in the coefficients
of elements of the approximating sequence. We thus obtain a sequence of tame automorphisms of the
skew Poisson algebra which converges to a limit which in turn is a conjugation of the initial
automorphism by a homothety-type transform. The application of Proposition \ref{skewthetaprop},
together with the obvious fact that the loop morphism maps sequence elements to themselves (due to
the fact that the latter are tame automorphisms), will conclude the proof.

We now proceed with the details. Suppose given an automorphism $\varphi\in \Aut P^h_{n,\mathbb{C}}$
of the $h$-augmented Poisson algebra. Without loss of generality, we may assume that the linear
part of $\varphi$ is the identity matrix: indeed, one can compose $\varphi$ with tame automorphisms
(tame approximation of automorphisms of $P^h_{n,\mathbb{C}}$ is valid according to an argument
similar to that of \cite{KGE}, as was mentioned previously), so that the linear part of the
resulting automorphism is the identity map; also the loop morphism $\Phi_s^h$ is point-wise stable
on tame automorphisms.

We add two more $h$-Poisson variables (and lift $\varphi$ to an automorphism of the new algebra by
demanding it be stable on the new generators) and, correspondingly, consider the skew Poisson
version -- the algebra $P_{n+1,\mathbb{C}}^h[k_{ij}]$ with the last two variables denoted by $u$
and $v$. Our immediate goal is to realise the algebra $P^h_{n,\mathbb{C}}$ as a subalgebra in the
appropriately localized version of $P_{n+1,\mathbb{C}}^h[k_{ij}]$. To that end, we consider the
algebra $P_{n+1,\mathbb{C}}^h[k_{ij}]$  and transform the main generators
$$
\lbrace \xi_1,\ldots, \xi_{2n},u,v\rbrace
$$
to
$$
\lbrace x_1,\ldots, x_{2n},u,v\rbrace
$$
with $\lbrace x_i, u\rbrace = 0$ and $\lbrace x_i, v\rbrace = 0$. The change of the generating set
is required to properly define the action of $\varphi$, so that it will be an automorphism and will
be in agreement with the conditions of Proposition \ref{skewthetaprop}. The variable change is done
according to
$$
x_i = \xi_i - \alpha_i u - \beta_i v
$$
with $\alpha_i = k_{i,2n+2}k_{2n+1,2n+2}^{-1}$ and $\beta_i = - k_{i,2n+1}k_{2n+1,2n+2}^{-1}$ for
$i=1,\ldots, 2n$. We extend the coefficient ring by adding the necessary variables. The new
generators $\lbrace x_1,\ldots, x_{2n}\rbrace$ commute according to
\begin{gather*}
\lbrace x_i, x_j\rbrace = h(k_{ij} - \alpha_jk_{i,2n+1}+\alpha_ik_{j,2n+1}-\beta_jk_{i,2n+2}+\\
\beta_ik_{j,2n+2}+ (\alpha_i\beta_j - \alpha_j\beta_i)k_{2n+1,2n+2}) = h\tilde{k}_{ij}.
\end{gather*}
Note that the new commutation relation matrix, which we denote by $[\tilde{k}_{ij}]$\footnote{We
exclude $u$ and $v$, so that $i$ and $j$ run from $1$ to $2n$.}, is again skew-symmetric, and that
its entries are $\mathbb{C}$-polynomial in the entries of the initial matrix and their inverses.

We now reduce the matrix $[\tilde{k}_{ij}]$ to the standard form (corresponding to the algebra
$P^h_n$) by transforming $\lbrace x_1,\ldots, x_{2n}\rbrace$ to \\$\lbrace q_1,\ldots,
q_n,p_1,\ldots, p_n\rbrace$ with
$$
\lbrace p_i, q_j \rbrace = h\delta_{ij}.
$$
The new variables $p_i$ and $q_j$ are expressed as linear combinations of $x_1,\ldots, x_{2n}$ with
coefficients in the appropriate polynomial ring.

The algebra $P^h_{n,\mathbb{C}}$ is therefore a subalgebra of the algebra generated by $$\lbrace
q_1,\ldots, q_n,p_1,\ldots, p_n, u,v\rbrace$$ (together with $h$ as the augmentation variable), as
the Poisson bracket takes its proper form after the standard form reduction, while $\mathbb{C}$ is
a subring of the coefficient ring.

We extend our automorphism $\varphi$ to act on this algebra: on $p_i$ and $q_j$ its action is given
by definition, and we impose $\varphi(u) = u$, $\varphi(v) = v$ and $\varphi(k_{ij}) = k_{ij}$.
Thus, starting from $\varphi$ we have arrived at an automorphism $\bar{\varphi}$ of the localized
skew Poisson algebra.

With respect to the initial generator set  $\lbrace \xi_1,\ldots, \xi_{2n},u,v\rbrace$ this
automorphism is generally not polynomial in $k_{ij}$, although it always will be polynomial in $h$.
In order to construct from it an automorphism of the skew algebra, we need to get rid of the
denominators first. This is accomplished by the following lemma.
\begin{lem}\label{conjugationlemma}
For every $\bar{\varphi}$ constructed as above, there is a polynomial $P$ in $k_{ij}$, such that
conjugation of $\bar{\varphi}$ with the transformation
$$
(\xi_1,\ldots, \xi_{2n},u,v)\mapsto (P\xi_1,\ldots, P\xi_{2n},Pu,Pv),\;\;h\mapsto P^2h
$$
is polynomial in $k_{ij}$.
\end{lem}
The proof of this Lemma is left to the interested reader, although it is fairly obvious why it must
be true. Indeed, the denominators in the expression for $\bar{\varphi}$ are polynomial in $k_{ij}$
coming from the separation of the $(u,v)$-plane and the standard form reduction (at which point the
determinant of $[\tilde{k}_{ij}]$ makes its contribution). One can therefore find appropriate
$P(k_{ij})$ to cancel these denominators. Furthermore, the polynomial $P$ depends only on the two
generator systems (or rather, on the transformation matrix between those).

\smallskip

We denote the result of the conjugation of Lemma \ref{conjugationlemma} by $\varphi^P$. The images
of the main generators (both in the cases of the initial -- skew -- generators as well as those
which correspond to the standard form) under $\varphi^P$ are, by Lemma \ref{conjugationlemma},
polynomial in $k_{ij}$, and are also by construction polynomial in $h$.

Now, the automorphism $\varphi^P$, when acting upon the standard form generators
$$
\lbrace q_1,\ldots, q_n,p_1,\ldots, p_n\rbrace
$$
can be viewed as an automorphism of the $h$-augmented Poisson algebra $P^h_{n,\mathbb{C}[\lbrace
hk_{ij}\rbrace]}$ over the polynomial ring $\mathbb{C}[\lbrace hk_{ij}\rbrace]$. The
$\mathbb{Z}$-grading of this algebra is specified by assigning degree $1$ to the main generators
and degree $0$ to $h$ and all $k_{ij}$. As an automorphism of this Poisson algebra, $\varphi^P$
admits, by an argument virtually identical to the main result of \cite{KGE}, a tame automorphism
(symplectomorphism) sequence converging to it in the power series topology induced by the above
grading. Let us fix such a sequence and denote it by $\lbrace \psi_m\rbrace$.

Every element $\psi_k$ of the tame sequence is such that the images under $\psi_m$ of the
generators $\lbrace q_1,\ldots, q_n,p_1,\ldots, p_n\rbrace$ are polynomial in $h$ and $k_{ij}$.
However, when acting upon the localized skew Poisson algebra generators $\lbrace \xi_1,\ldots,
\xi_{2n}\rbrace$, $\psi_m$ need not be polynomial in $k_{ij}$, and therefore $\psi_m$ cannot be
lifted to automorphisms of the skew Poisson algebra. This is remedied by application of Lemma
\ref{conjugationlemma}: one can find a polynomial $P_1$ in the variables $k_{ij}$, such that the
conjugation of every element $\psi_m$ of the tame sequence with the mapping
$$
(\xi_1,\ldots, \xi_{2n},u,v)\mapsto (P_1\xi_1,\ldots, P_1\xi_{2n},P_1u,P_1v),\;\;h\mapsto P_1^2h
$$
yields an automorphism of the localized skew Poisson algebra polynomial in $k_{ij}$. We then have
the following statement.
\begin{lem}\label{convergencelemma}
The sequence $\lbrace \psi^{P_1}_m\rbrace$ converges to the conjugated automorphism
$(\varphi^P)^{P_1}$ in the power series topology with $\Deg h = \Deg k_{ij} = 0$ as well as in the
power series topology with $\Deg h = 0$, $\Deg k_{ij} = 2$.
\end{lem}
\begin{proof}
The first half of the statement follows from the construction of the tame sequence and from the
observation that, due to the fact that the two coordinate systems are connected by a transformation
that has zero free term, the height of the polynomials $P$ and $P_1$ is at least one.

One then obtains convergence in the power series topology relevant to the singularity trick
(Proposition \ref{skewthetaprop}) from that in the approximation power series topology by noting
that giving a non-zero degree to $k_{ij}$ may only make the consecutive approximations closer to
the limit in the corresponding metric.

\end{proof}

The sequence $\lbrace \psi^{P_1}_m\rbrace$ can, due to its polynomial character with respect to
$k_{ij}$, be thought of as a sequence of tame automorphisms of the skew Poisson algebra
$P_{n+1,\mathbb{C}}^h[k_{ij}]$ over $\mathbb{C}$ converging to $(\varphi^P)^{P_1}$. By Proposition
\ref{skewthetaprop}, the loop morphism $\Phi_s^{hk}$ is continuous and by the theorem of \cite{4}
regarding tame automorphisms (or, to be more precise, by its counterpart for the case of skew Weyl
and Poisson algebras), $\Phi_s^{hk}(\psi^{P_1}_m) = \psi^{P_1}_m$. Therefore,
$$
\Phi_s^{hk}((\varphi^P)^{P_1}) = (\varphi^P)^{P_1}.
$$
Finally, as $(\varphi^P)^{P_1}$ is obtained from $\varphi$ by conjugation with homothety-type
transformations, and as the mappings $\phi_{[p]}$ and their $h$-augmented and skew analogues, being
in essence compositions of restrictions and Frobenius homomorphisms, manifestly preserve such
transformations point-wise, we conclude that
$$
\Phi_s^h(\varphi) = \varphi.
$$
Proposition \ref{mainprop} is proved.

\subsection{Returning to the non-augmented case. Automorphisms Laurent in $h$.}
In order to return to the non-augmented algebras (i.e. to the context of the original Conjecture
\ref{mainconj}), we need to specialize to $h = 1$. The argument is essentially similar to the one
needed in the proof of Conjecture \ref{mainconj} in \cite{K-BE4}, yet the situation is slightly
simpler here as we need not establish convergence (in fact, truncation) of power series in $h$.

We need to show how the fact that $\Phi^h_s$ is the identity map implies that $\Phi_s$, the
original loop morphism (for the Poisson algebra), is also the identity map. To that end, we observe
that the augmentation by $h$ along with the main results of the previous subsection, admit the
following reformulation.

Namely, we notice that we may regard $h$ as scalar coefficients (i.e. coefficients in large enough
a base field) and consider the (Poisson) $\mathbb{C}(h)$-algebra obtained from $P^h_{n,\mathbb{C}}$
by preserving the form of the augmented Poisson bracket. The new ground field $\mathbb{C}(h)$
carries an obvious $h$-adic topology.

The fact that we need $\mathbb{C}(h)$ -- or, to be more precise, the ring $\mathbb{C}[h,h^{-1}]$ of
Laurent polynomials -- rather than the mere polynomial ring $\mathbb{C}[h]$ is dictated by the
necessity to make specialization of $\Aut P^h_{n,\mathbb{C}}$ to $\Aut P_{n,\mathbb{C}}$
surjective. This can be seen as follows. Suppose $\varphi^h$ is an automorphism of the deformed
algebra $P^h_{n,\mathbb{C}}$ which acts as the identity map on $h$. Since it is stable on $h$, it
corresponds to an automorphism of the $\mathbb{C}[h]$-algebra (where $h$ is a parameter and not a
generator) generated by $x_i,\;p_j$ with the Poisson bracket containing $h$. This object, after
appropriate localization maps to an automorphism of the (augmented) Poisson algebra $P^h_n$ with
the ground field $\mathbb{C}(h)$. On the other hand, any automorphism $\varphi$ of
$P_{n,\mathbb{C}}$ can be made into a $\mathbb{C}(h)$-automorphism $\varphi^h$ by introducing a
scalar $h$ and conjugating $\varphi$ with a mapping
$$
x_i' = hx_i,\;\;p_j' = p_j.
$$
The resulting transformation will be an automorphism of the Poisson $\mathbb{C}(h)$-algebra with
the bracket as in the augmented algebra $P^h_n$,\emph{ however in general the images of the
generators under this automorphism will contain negative powers of $h$.  } Its specialization to
$h=1$ returns it to $\varphi$. Therefore, every polynomial symplectomorphism has a pre-image under
specialization of automorphisms rational in the augmentation parameter. The actual Laurent
dependence is due to the fact any $\varphi$ is an object of finite degree, therefore its
conjugation by the transformation above will lead to an object polynomial in $h$ and $h^{-1}$. We
nevertheless refer to these objects as automorphisms over the field $\mathbb{C}(h)$, with the
additional assumption of their polynomial dependence on $h$ and $h^{-1}$.

\smallskip

The conclusion is that Proposition\ref{mainprop} does not immediately imply that the non-augmented
loop morphism $\Phi_s$ is the identity map; rather, the domain of the loop morphism $\Phi^h_s$ needs
to be extended to the points with rational dependency on the augmentation parameter, at which point
the claim $\Phi_s=\Id$ follows from the specialization of the extended loop morphism $\overline{\Phi^h_s}$.

\smallskip

The problem we have encountered is identical to the one resolved in our proof of the isomorphism in
\cite{K-BE4} and will be processed here in a similar manner. Namely, we are going to introduce
auxiliary variables and construct a so-called \emph{twisted} symplectomorphism from a fixed
symplectomorphism which is polynomial in $h$ and $h^{-1}$ (thus, one must verify that the action of
the loop morphism on this $\mathbb{C}(h)$-algebra automorphism is that of the identity map). The
twisted object will be polynomial in $h$ and therefore the action of $\Phi^h_s$ on it will be well
defined; we will obtain the desired equality once we compare the images of the auxiliary variable
under the twisted object and its image under the loop morphism.

\smallskip

Let $\varphi$ be an automorphism of the augmented Poisson $\mathbb{C}(h)$-algebra polynomial in $h$
and $h^{-1}$. We introduce a pair of auxiliary variables, $u$ and $v$, which are extra $x$ and $p$
with respect to the augmented Poisson bracket. Let
$$
\lambda = h^k
$$
be $k$-th power of the augmentation parameter, for large enough $k$. We fix $i$ (the number of the
pair $x_i,\;p_i$) and we define the automorphism
$$
\psi_{\lambda}: u\mapsto u + \lambda x_i,\;\;p_i\mapsto p_i - \lambda v.
$$
As before, we extend $\varphi$ to the new algebra by its identical action on the auxiliary
variables and denote the extended map also by $\varphi$. We now consider the twisted automorphism
$$
\varphi_{\lambda} = \varphi\circ\psi_{\lambda}\circ\varphi^{-1}.
$$
As $k$ can be taken arbitrarily large, the mapping $\varphi_{\lambda}$ will be polynomial in $h$
for all $k>k_0$ (where $k_0$ depends on $\varphi$ as well as $\varphi^{-1}$ but is finite for the
fixed automorphism).

We are now able to apply the loop morphism $\Phi^h_s$ to the twisted automorphism $\varphi_{\lambda}$,
by Proposition \ref{mainprop}, the output must be $\varphi_{\lambda}$ itself. To finalize the proof,
we need to compare the action of $\Phi^h_s(\varphi_{\lambda})$ on $u$ with $\varphi_{\lambda}(u)$,
however in order to do so we must know the form of the image. The next two lemmas address the issue.


\begin{lem}\label{twistedliftcanon}
Suppose $\theta$ is an $h$-augmented polynomial symplectomorphism over $\mathbb{C}$. Denote by
$\lbrace \theta_p\rbrace$ the sequence of characteristic $p$ symplectomorphisms representing its
modulo $[p]$ reduction. For a generic element $p$ in a sequence representing $[p]$, denote the Weyl
generators by $x_1,\ldots, x_n, d_1,\ldots, d_n$ and the corresponding $p$-th powers generating the
center of the Weyl algebra over $\mathbb{F}_p$ by  $\xi_1,\ldots, \xi_n, \eta_1,\ldots, \eta_n$.
Then, for almost all $p$ in $[p]$ (in the sense of the ultrafilter), the image under $\theta_p$ of
every central generator admits a unique pre-image Weyl polynomial $\hat{H}$ with respect to taking
the $p$-th power and pulling back the coefficients by the inverse Frobenius automorphism.
\end{lem}
\begin{proof}
We prove the statement for $H = \theta_p(\xi_i)$ -- the case $\eta_j$ is identical.

Suppose first that
$$
\theta_p(\xi_i) = \xi_i = x_i^p.
$$
Then the Newton polyhedron of the image $\theta_p(\xi_i)$ has only one vertex, therefore -- as
taking the $p$-th power only dilates the Newton polyhedron -- the polynomial $\hat{\theta}_p(x_i)$
must be equal to $x_i$.

The general case uses the isomorphism between the automorphism groups of $h$-augmented algebras,
which implies, given its nature, that modulo $[p]$ reductions of $\theta$ and its lifting
$\hat{\theta}$ are consistent -- that is, for almost all $p$ in $[p]$, the restriction of
$\hat{\theta}_p$ to the center (twisted by the inverse Frobenius acting on the coefficients)
coincides with $\theta_p$. The application is as follows. Suppose
$$
H = \theta_p(\xi_i)
$$
is the image of $\xi_i$.  We know that
$$
H = \Fr_*^{-1}\hat{\theta}_p(x_i^p)
$$
where $\Fr_*^{-1}$ is the action of the inverse Frobenius automorphism on the coefficients of the
polynomial. The last equation is equivalent to
$$
\hat{\theta}_p^{-1}(\Fr_*(H)) = x_i^p.
$$
By the special case above, there exists a unique Weyl polynomial $\hat{G}$ such that
$$
\hat{G}^p = \hat{\theta}_p^{-1}(\Fr_*(H)).
$$
But then
$$
H = \Fr_*^{-1}(\hat{\theta}_p(\hat{G}^p))
$$
which is exactly what we wanted.
\end{proof}

It will be convenient to denote the one-to-one correspondence between modulo $p$ reductions of
central polynomials coming from characteristic zero symplectomorphisms with their Weyl liftings by
$\phi^h_p$ (for this correspondence is, as evidenced by Lemma \ref{twistedliftcanon}, shares
essential nature with the characteristic zero direct homomorphism $\phi^h_{[p]}$).

We now apply the above lemma in order to establish the form of the pre-image Weyl polynomial in the
case of auxiliary variables $u,v$ and the central polynomial of a special type.
\begin{lem}\label{formlemma}
Let $u,v$ denote two extra $h$-Poisson variables, and let
$$
H = u + h^k\varphi(x_i)
$$
be the image of $u$ under the twisted automorphism coming from $\varphi$ as above ($\varphi(x_i)$
is rational in $h$ but $h^k\varphi(x_i)$ is polynomial in $h$). Then the unique pre-image $\hat{H}$
of $H$ with respect to the correspondence $\phi^h_p$ of the previous lemma has the form
$$
\hat{H} = \hat{u} + h^k P_i(x_1,\ldots, d_n,h)
$$
where $P_i$ is rational in $h$.
\end{lem}
\begin{proof}
We establish the statement in several elementary steps. Firstly, as $H$ does not contain the
auxiliary variable $v$, $\hat{H}$ does not contain its Weyl counterpart $\hat{v}$: indeed,
otherwise the Newton polyhedron of $H$ would contain (in the case of $v$ carrying great enough
weight to make the corresponding monomial the highest-order term) a vertex corresponding to the
monomial containing $\hat{v}$. \footnote{Note that $\phi^h_p$ behaves toward the Newton polyhedra
as the homomorphism taking the $p$-th power does.}

Now let
$$
\hat{H} = Q(\hat{u}) + R
$$
where every monomial in $R$ is proportional to generators other than $u$. Then
$$
H = \phi^h_p(\hat{H}) = \phi^h_p(Q) + \phi^h_p(R),
$$
as the two differential operators $Q$ and $R$ commute with each other and therefore taking the
$p$-th power is executed as in the commutative case. By Lemma \ref{twistedliftcanon}, we must have
$$
Q(\hat{u}) = \hat{u}.
$$

Finally, we show that if $\hat{H}$ contains monomials which are products of $\hat{u}$ with other
generators, then  $\phi^h_p(\hat{H})\neq H$.  Indeed, if such a monomial had a non-zero coefficient
in  $\hat{H}$, then there would exist a grading under which this monomial would be the
highest-order term (corresponding to a vertex in the Newton polyhedron). Then the image
$\phi^h_p(\hat{H})$ would also have a monomial corresponding to this highest-order term with
non-zero coefficient, as taking the $p$-th power dilates the polyhedron and therefore maps the
extremal points to extremal points.

The conclusion is that the polynomial $\hat{H}$ has the form
$$
\hat{u} + \tilde{P}_i(x_1,\ldots, d_n,h).
$$
Taking out $h^k$ from $\tilde{P}_i$ leaves us with the form we needed.
\end{proof}

The two lemmas provide a canonical way to extend the $h$-augmented symplectomorphism lifting map to points
polynomial in $h$ and $h^{-1}$ -- the expressions $P_i(x_1,\ldots, d_n,h)$ (and their analogues for $p_j$)
are taken to be the images of the lifted automorphism, at which point the commutation relations are verified as in \cite{K-BE4}.
As the polynomials $P_i(x_1,\ldots, d_n,h)$ define the lifting and are therefore tied to the extension of the Kontsevich isomorphism
$\phi^h_{[p]}$ to $h$-Laurent points, we conclude that the action of the loop morphism $\Phi^h_s$ on $\varphi_{\lambda}$ yields
$$
\Phi^h_s(\varphi_{\lambda})(u) = u + \lambda \overline{\Phi^h_s}(\varphi),
$$
where $\overline{\Phi^h_s}(\varphi)$ is the desired extension of the loop morphism to points Laurent-polynomial in $h$.

The last equation, together with Proposition \ref{mainprop}, implies that
$$
u + \lambda\overline{\Phi^h_s}(\varphi)(x_i) = u + \lambda \varphi(x_i),
$$
from which it follows that the action of $\varphi$ and $\overline{\Phi^h_s}(\varphi)$ is identical
on each generator $x_i$. The other group of generators ($p_i$) is processed analogously.

We have thus proved the following result.
\begin{prop}\label{mainpropbar}
$\overline{\Phi^h_s}$ is the identity map.
\end{prop}

The Main Theorem is an easy consequence of Proposition \ref{mainpropbar}, as the specialization of
$\overline{\Phi^h_s}$ is $\Phi_s$, and every point in the domain of $\Phi_s$ has a pre-image under
the specialization.

\begin{remark}\label{remnagata}
This form of specialization argument, particularly the switching from automorphisms preserving $h$
to automorphisms of $\mathbb{C}(h)$-algebras is similar to the method used by Nagata \cite{Nagata}
in the construction of his famous example (also cf. \cite{Shes1, KBYu_N1, KBYu_N2}).
\end{remark}

\begin{remark}\label{simplerproof}
There is a slightly shorter justification of Proposition \ref{mainpropbar}. Namely, from the extension of the lifting map in \cite{K-BE4}
to points polynomial in $h$, $h^{-1}$, where Lemmas \ref{twistedliftcanon} and \ref{formlemma} are essential and have first appeared, it follows that $\phi^h_{[p]}$ extends to an isomorphism over
$\mathbb{C}[h,h^{-1}]$. Thus (as before) the loop morphism is well defined for such objects. However now we know that the $\overline{\Phi^h_s}$ is a group homomorphism; therefore, when starting with $\varphi$
and constructing the twisted automorphism, we may immediately write
$$
\Phi^h_s(\varphi\circ\psi_{\lambda}\circ\varphi^{-1}) = \overline{\Phi^h_s}(\varphi)\circ \psi_{\lambda}\circ \overline{\Phi^h_s}(\varphi)^{-1}
$$
after which the evaluation of the image of $u$ becomes a trivial affair.
\end{remark}

\section{Discussion}

As we have seen, the introduction of skew Poisson and Weyl commutator structure augmentation, which
eventually leads to a stronger form of tame approximation proved via the singularity trick, with
further specialization of augmentation parameters, allows one to arrive at the Main Theorem.

\smallskip

The introduction of the deformation (augmentation) parameter $h$ and the skew form $[k_{ij}]$
allows in fact to construct a well-behaved lifting map from (skew augmented) symplectomorphisms to
automorphisms of (completion of) the skew Weyl algebra, which is the inverse to the direct
homomorphism $\phi^h_{[p]}$. This is the framework for the proof of Conjecture \ref{mainconj},
conducted in detail in \cite{K-BE4}. The specialization to $h=1$ is achieved by first showing that
it is correct in an irreducible component and then by using the homotopy argument as demonstrated
in the main text. The added complexity lies in the issue of power series convergence, which means
that the lifting map has to be continued to the entire domain. This complication does not occur
here as Conjecture \ref{mainconj} has been assumed.
\smallskip

The nice behavior of the skew version of the homomorphism $\phi_{[p]}$ under approximation seems to
be a rather potent result. It is a question of legitimate interest whether the endomorphism
counterpart of this map is algebraic. The algebraic nature in the automorphism case was established
via Gabber's idea, utilizing local nilpotency of derivations corresponding to the adjoint action by
generators. It may be possible to modify this attack to arrive at the endomorphism case.
Furthermore, if one abandons the loop morphism point of view in favor of comparing, for various
infinite primes, the action of $\phi_{[p]}$ and its augmented and skew analogues, one can -- as we
plan to show in a separate paper -- effectively circumvent certain issues pertaining to regularity
and approximation and mount a successful attack on the (augmented) endomorphism case of the
independence problem.

The non-augmented endomorphism counterpart to the Main Theorem seems to require the following
rather curious Conjecture \ref{irrconj}.

\begin{conj}\label{irrconj}
Let $\End W_{n,\mathbb{C}}$ denote the $\Ind$-scheme of endomorphisms of the Weyl algebra, which is
the direct limit of the system of schemes $\End^{\leq N}$. Then every irreducible component of each
scheme $\End^{\leq N}$ is embedded into an irreducible component of $\End^{\leq M}$ ($M\geq N$)
which contains the point corresponding to the identity automorphism.
\end{conj}
Its resolution would constitute a significant advancement in our understanding of the geometry and
topology underlying endomorphism $\Ind$-schemes.

\smallskip

Finally, as it has been outlined in the overview of main results, the proof strategy (the
augmentation, tame augmented approximation, and specialization of augmentation parameters) is not
sensitive to the base field (as long as the latter has characteristic zero). In the case of base
field $\mathbb{Q}$, thanks to Proposition \ref{skewthetaprop}, limits of the image of convergent
sequences of tame (skew augmented Weyl algebra) automorphisms under the tame isomorphism -- which
is manifestly independent of infinite prime -- exist and depend only on the limits of tame
sequences. The conclusion is that the results of \cite{4}, together with the constructions
presented here as well as in \cite{K-BE4}, eventually yield the proof of the general Kontsevich
conjecture (Conjecture \ref{mainconjgen}), as well as the independence of infinite prime over
arbitrary field of characteristic zero.

\end{document}